\documentclass[a4paper]{article}

\usepackage[top=2cm, bottom=2cm, left=1.8cm, right=1.8cm]{geometry}
\usepackage{amsfonts}
\usepackage{latexsym}
\usepackage{amsmath}
\usepackage{setspace}
\usepackage{comment}
\usepackage{bm}
\usepackage{graphics}
\usepackage{graphicx}
\usepackage{epstopdf}
\usepackage[caption=false]{subfig}
\usepackage{color}
\usepackage{algorithm}
\usepackage{algpseudocode}
\usepackage{algorithmicx}
\usepackage{amsopn}
\usepackage{amssymb}
\usepackage{mathtools}
\usepackage{amsthm}
\usepackage{booktabs}

\newcommand{\mini}{\mathop{\mbox{minimize}}}
\newcommand{\subj}{\mathop{\mbox{subject to}}}
\newcommand{\dist}{{\rm dist}}

\theoremstyle{definition}
\newtheorem{definition}{Definition}
\newtheorem{theorem}{Theorem}
\newtheorem{proposition}{Proposition}
\newtheorem{corollary}{Corollary}
\newtheorem{lemma}{Lemma}

\begin{document}

\title{Local convergence analysis of stabilized sequential quadratic programming methods for optimization problems in Banach spaces}

\author{Yuya Yamakawa\thanks{Graduate School of Management, Tokyo Metropolitan University, Tokyo 192-0397, Japan, E-mail:~yuya@tmu.ac.jp}}


\maketitle

\begin{abstract}
This paper presents a stabilized sequential quadratic programming (SQP) method for solving optimization problems in Banach spaces. The optimization problem considered in this study has a general form that enables us to represent various types of optimization problems and is particularly applicable to optimal control, obstacle, and shape optimization problems. Several SQP methods have been proposed for optimization problems in Banach spaces with specific structures; however, research on the local analysis of SQP-type methods for general problems, such as those considered in this study, is limited. We focus on the local behavior of the proposed stabilized SQP method and prove its local quadratic convergence under reasonable assumptions.
\end{abstract}

\section{Introduction} \label{sec:intro}
We consider the following optimization problems in Banach spaces:
\begin{align} \label{problem:main}
\begin{aligned}
& \displaystyle \mini_{x \in X} & & f(x)
\\
& \subj & & G(x) \in K,
\end{aligned}
\end{align}
where $X$ is a real Banach space, $Y$ is a real Hilbert space, $K$ is a nonempty closed convex set of $Y$, and $f \colon X \to \mathbb{R}$ and $G \colon X \to Y$ are twice continuously differentiable on $X$.
\par
Problem~\eqref{problem:main} provides a general framework that unifies the various optimization problems and plays a crucial role in describing optimization problems in Banach spaces. Such abstract optimization problems arise in several fields, including optimal control, obstacle, and shape optimization problems. In these fields, several researchers have studied various types of optimization problems and proposed methods for solving them~\cite{ArRaTr02,Be93,BeKu97,BlRu17,BoKaSt19,ClDeMePr20,FaVaMe23,GaHeHiKa15,HiHi06,HiKu06,HoNe21,HoDi13,ItKu96,KaVo12,KaSt18,KaStWa18,KaKuTa03,KlSa97,Ku96,PrTrWe08,St09,UlUl09,Vo00,Wa07,Ya23,ZiUl11}.  However, research on methods for optimization problems with a general form such as~\eqref{problem:main} is limited.
\par
Sequential quadratic programming (SQP) methods are widely known as one of the most powerful iterative methods for constrained optimization problems. These methods generate sequences by iteratively solving quadratic programming subproblems that approximate the objective function quadratically and the constraint functions linearly. In particular, many types of SQP methods, which are collectively known as SQP-type methods, have been developed for finite-dimensional optimization problems, including nonlinear programming~\cite{BoLa95,BoTo96,GiRo13,Ha99,IzSoUs15,So09,Sp98,Wr98}, nonlinear semidefinite programming~\cite{CoRa04,FaNoAp02,LiZh19,OkYaFu23,YaOk22,ZhZh16}, and other problems~\cite{HoDa00,KaFu07,LiPeSu06}. Their local fast convergence properties have been studied extensively. Meanwhile, although a significant number of SQP-type methods have been proposed for infinite-dimensional optimization problems~\cite{FaVaMe23,HoNe21,HoDi13,Ku96,Vo00,Wa07,Ya23}, they remain relatively fewer in number compared with those for finite-dimensional problems. Moreover, a limited number of studies have shown the local convergence of SQP-type methods for optimization problems with a general form such as~\eqref{problem:main}, and this limitation is not exclusive to SQP-type methods.
\par
Existing studies on optimization methods for problems in Banach spaces with a wide range of applications, such as~\eqref{problem:main}, can be found specifically in~\cite{BoKaSt19,KaSt18,Ya23}. Kanzow et al.~\cite{KaSt18} presented an augmented Lagrangian (AL) method for variational problems in Banach spaces. Variational problems encompass the first-order necessary conditions for~\eqref{problem:main}, which implies that the AL method could be applied to~\eqref{problem:main}. They established the local linear convergence of the AL method under the strict Robinson's constraint qualification (SRCQ), the second-order sufficient condition (SOSC), and an assumption that any generated sequence converges to some feasible point. B\"{o}rgens et al.~\cite{BoKaSt19} proposed another AL method for optimization problems in Banach spaces and provided its local convergence analysis. The problems considered in~\cite{BoKaSt19} are more general than~\eqref{problem:main} as the space $Y$ is a Banach space continuously embedded in some Hilbert space, and therefore the AL method has wider applications than that proposed in~\cite{KaSt18}. However, their local convergence analysis established only the strong convergence of the primal iterate, excluding the Lagrange multiplier, under a stronger SOSC than that of~\cite{KaSt18} and did not address the convergence rate. Yamakawa~\cite{Ya23} presented a stabilized SQP method for solving optimization problems in function spaces and proved its global convergence. The problems considered in~\cite{Ya23} are formulated in a broadly applicable form to describe optimization in function spaces, whereas they are included in~\eqref{problem:main} as a special case. Moreover, although the study analyzed the global convergence of the stabilized SQP method, it did not address its local and fast convergence.
\par
This study aims to develop a stabilized SQP method for solving~\eqref{problem:main} and prove its local quadratic convergence under reasonable assumptions. To the best of the author's knowledge, although Yamakawa~\cite{Ya23} proposed the globally convergent stabilized SQP method, no research has been conducted on the local fast convergence of stabilized SQP methods for problems with the same general form as~\eqref{problem:main}. This study is believed to be the first to prove the local quadratic convergence of stabilized SQP methods in this setting. The local convergence analysis of the proposed method assumes the SRCQ and the SOSC, and does not require any additional conditions, such as the convergence of a generated sequence to some feasible point, as assumed in~\cite{KaSt18}. Therefore, this study yields better results than those in~\cite{KaSt18}, as it establishes an optimization method with a higher convergence rate under reasonable assumptions. Moreover, in the field of finite-dimensional optimization, some fast convergent stabilized SQP methods have been established under milder assumptions, not including any constraint qualification (CQ). Their fast convergence properties are mainly derived from the fact that their subproblems can be solved exactly at all iteration points, and they satisfy the Lipschitz stability. Meanwhile, for infinite-dimensional optimization, even the existence of solutions to the subproblems is not guaranteed in the first place. Hence, it is difficult to directly extend the same technique as the existing study on stabilized SQP methods for finite-dimensional optimization to this study. However, under the SRCQ and the SOSC, we can show that the subproblems have approximate solutions that play a crucial role in the local convergence analysis. By utilizing them, the fast convergence of the proposed method can be established in a different manner from the existing study, and this is one remarkable point in the local convergence analysis.
\par
The remainder of the paper is organized as follows: Section~\ref{sec:pre} introduces mathematical notation and terminologies related to necessary and sufficient optimality conditions for problem~\eqref{problem:main}. In Section~\ref{sec:SSQP}, we review SQP-type methods and propose a stabilized SQP method for solving~\eqref{problem:main}. Finally, Section~\ref{sec:local_convergence} proves the local quadratic convergence property of the stabilized SQP method.

\section{Preliminaries} \label{sec:pre}
This section introduces mathematical notation and terminologies regarding optimality conditions for problem~\eqref{problem:main}. 

\subsection{Basic mathematical notation}
We denote the set of positive integers by $\mathbb{N}$. Let ${\cal X}$ and ${\cal Y}$ be real Banach spaces, and let ${\cal L}({\cal X},{\cal Y})$ be the normed space of all bounded linear operators from ${\cal X}$ to ${\cal Y}$. Moreover, we define ${\cal X}^{\ast} \coloneqq {\cal L}({\cal X}, \mathbb{R})$. The norms on ${\cal X}$ and ${\cal L}({\cal X}, {\cal Y})$ are defined by $\Vert \cdot \Vert_{{\cal X}}$ and $\Vert \cdot \Vert_{{\cal X} \to {\cal Y}}$, respectively. We denote the adjoint operator of $\varphi \in {\cal L}({\cal X}, {\cal Y})$ by $\varphi^{\ast} \in {\cal L}({\cal Y}^{\ast}, {\cal X}^{\ast})$. Let $\langle \cdot, \cdot \rangle_{{\cal X}^{\ast}, {\cal X}}$ be the associated dual pairing. If ${\cal X}$ is a Hilbert space, then its inner product is denoted by $(\cdot, \cdot)_{{\cal X}}$, its norm is defined by $\Vert \cdot \Vert_{{\cal X}} \coloneqq \sqrt{( \cdot, \cdot )_{{\cal X}}}$, and its dual ${\cal X}^{\ast}$ is identified with ${\cal X}$. Let ${\cal Z} \coloneqq {\cal X} \times {\cal Y}$ be the product space. The norm on ${\cal Z}$ is defined by $\Vert z \Vert_{{\cal Z}} \coloneqq \Vert x \Vert_{{\cal X}} + \Vert y \Vert_{\cal Y}$ for $z = (x, y) \in {\cal Z}$. Let ${\cal K}$ and ${\cal L}$ be sets of ${\cal X}$. We define ${\cal K} + {\cal L} \coloneqq \{ \alpha + \beta \in {\cal X}; \alpha \in {\cal K}, \beta \in {\cal L} \}$. For any number $\gamma \in \mathbb{R}$, we define $\gamma {\cal K} \coloneqq \{ \gamma \alpha; \alpha \in {\cal K} \}$. The polar cone of ${\cal K}$ is defined by ${\cal K}^{\circ} \coloneqq \{ \varphi \in {\cal X}^{\ast} ; \langle \varphi, \alpha \rangle_{{\cal X}^{\ast}, {\cal X}} \leq 0 ~ \forall \alpha \in {\cal K} \}$. For any point $x \in {\cal X}$, we define $\dist(x, {\cal K}) \coloneqq \inf \{ \Vert x - \alpha \Vert_{X}; \alpha \in {\cal K} \}$. The closed ball with center $c \in {\cal X}$ and radius $r > 0$ is denoted by $B_{{\cal X}}(c,r) \coloneqq \{ x \in {\cal X}; \, \Vert x - c \Vert_{{\cal X}} \leq r \}$. The closed ball with center at the origin is denoted by $B_{{\cal X}}$ for simplicity. For any nonempty convex set $C \subset {\cal X}$ and point $x \in C$, the tangent and normal cones of $C$ at $x$ is defined by ${\cal T}_{C}(x) \coloneqq \{ p \in {\cal X}; \, \exists \{ x_{k} \} \subset C, \, \exists \{ t_{k} \} \subset \mathbb{R} ~ \mbox{s.t.} ~ t_{k} > 0 ~ \forall k \in \mathbb{N}, \, \lim_{k \to \infty} t_{k} = 0, \, \lim_{k \to \infty} t_{k}^{-1}(x_{k} - x) = p \}$ and ${\cal N}_{C}(x) \coloneqq \{ \varphi \in {\cal X}^{\ast}; \langle \varphi, p - x \rangle_{{\cal X}^{\ast}, {\cal X}} \leq 0 ~ \forall p \in C \}$, respectively. Let ${\cal F} \colon {\cal X} \to {\cal Y}$ be a Fr\'echet differentiable functional at $x \in {\cal X}$. We denote the Fr\'echet derivative of ${\cal F}$ by ${\cal F}^{\prime}$. If ${\cal X}$ is a product space such that ${\cal X} = {\cal X}_{1} \times \cdots \times {\cal X}_{n}$ with $n \geq 2$, then $x \in X$ is represented as $x = (x_{1},\ldots,x_{n}) \in {\cal X}_{1} \times \cdots \times {\cal X}_{n}$, the partial Fr\'echet derivative of ${\cal F}$ with respect to $x_{i} \in {\cal X}_{i}$ is denoted by ${\cal F}_{x_{i}}$, and the partial Fr\'echet derivative of ${\cal F}_{x_{i}}$ with respect to $x_{j} \in {\cal X}_{j}$ is denoted by ${\cal F}_{x_{i}x_{j}}$. For a multifunction $\Psi \colon {\cal X} \to 2^{{\cal Y}}$ from X to the set $2^{{\cal Y}}$ of all subsets of $Y$, its graph inverse $\Psi^{-1} \colon {\cal Y} \to 2^{{\cal X}}$ is defined as $\Psi^{-1}(y) \coloneqq \{ x \in {\cal X}; y \in \Psi(x) \}$. For any number $\varepsilon \geq 0$, set ${\cal S} \subset {\cal X}$, and function $\psi \colon {\cal S} \to \mathbb{R}$, we say that $\bar{x}$ is an $\varepsilon$-optimal solution of an optimization problem that minimizes $\psi$ on ${\cal S}$ if $x \in {\cal S}$ and $\psi(\bar{x}) \leq \inf \{ \psi(x); x \in {\cal S} \} + \varepsilon$. If an infinite sequence $\{ x_{k} \} \subset {\cal X}$ strongly converges to some $x \in {\cal X}$, namely, $\Vert x_{k} - x \Vert_{{\cal X}} \to 0~(k \to \infty)$, then we write $x_{k} \to x~(k \to \infty)$. For a closed convex set $C$ in a Hilbert space, the metric projector over $C$ is denoted by $P_{C}$. For a set $T$ in a topological space, its interior and cardinality are denoted by ${\rm int}(T)$ and ${\rm card}(T)$, respectively.

\subsection{Optimality conditions}
To begin with, we define $\Phi \coloneqq \{ x \in X; G(x) \in K \}$ and $V \coloneqq X \times Y$. Moreover, the Lagrange function $L \colon V \to \mathbb{R}$ is defined as follows:
\begin{align*}
L(v) \coloneqq f(x) + ( \lambda, G(x) )_{Y}, \quad v = (x, \lambda).
\end{align*}
The Karush-Kuhn-Tucker (KKT) conditions for problem~\eqref{problem:main} are defined as follows:
\begin{definition}
We say that $v = (x, \lambda) \in V$ satisfies the KKT conditions of problem~\eqref{problem:main} if
\begin{align*}
L_{x}(v) = 0, \quad \lambda \in {\cal N}_{K}(G(x)).
\end{align*}
Moreover, we say $v$ the KKT point of problem~\eqref{problem:main}.
\end{definition}
\noindent
The KKT conditions provide the first-order necessary optimality for~\eqref{problem:main} under regularity conditions, referred to as constraint qualifications (CQs). Several types of CQs have been proposed so far. This study presents two types of CQs given below.
\begin{definition}
We say that the Robinson constraint qualification (RCQ) holds at $x \in \Phi$ if
\begin{align*}
0 \in {\rm int}(G(x) + G^{\prime}(x) X - K).
\end{align*}
\end{definition}

\begin{definition}
Let $v = (x, \lambda) \in X \times Y$ be a KKT point of problem~\eqref{problem:main}. We say that the strict RCQ (SRCQ) holds at $v$ if
\begin{align*}
0 \in {\rm int}(G(x) + G^{\prime}(x) X - K_{0}),
\end{align*}
where $K_{0} \coloneqq \{ y \in K; (\lambda, y - G(x))_{Y} = 0 \}$.
\end{definition}
\noindent
In the subsequent argument, let $\bar{x} \in X$ be a local optimum of~\eqref{problem:main}. Moreover, the set $\Lambda(\bar{x})$ is defined as
\begin{align*}
\Lambda(\bar{x}) \coloneqq \{ \lambda \in Y; f^{\prime}(\bar{x}) + G^{\prime}(\bar{x})^{\ast} \lambda = 0, \lambda \in {\cal N}(G(\bar{x})) \}.
\end{align*}
Note that $\Lambda(\bar{x})$ is the set of Lagrange multipliers $\lambda \in Y$ such that $(\bar{x}, \lambda)$ satisfies the KKT conditions. If the RCQ holds at $\bar{x}$, then $\Lambda(\bar{x})$ is a nonempty, convex, bounded, and weakly compact subset of $Y$. Meanwhile, the SRCQ clearly implies the RCQ and further ensures that $\Lambda(\bar{x}) = \{ \bar{\lambda} \}$. Further details on these properties are provided in~\cite[Theorems~3.9 and 4.47]{BoSh00}. 
\par
Next, we introduce the second-order optimality. The local analysis discussed in Section~\ref{sec:local_convergence} utilizes the following SOSC.
\begin{definition} \label{def:SOSC}
Let $v = (x, \lambda) \in X \times Y$ satisfy the KKT conditions of~\eqref{problem:main}. We say that the SOSC holds at $v$ if there exist $c > 0$ and $\eta > 0$ such that 
\begin{align*}
\langle L_{xx}(v) d, d \rangle_{X^{\ast}, X} \geq c \Vert d \Vert_{X}^{2}
\end{align*}
for all $d \in C_{\eta}(x)$, where
\begin{align*}
C_{\eta}(x) \coloneqq \{ d \in X; \langle f^{\prime}(x), d \rangle_{X^{\ast}, X} \leq \eta \Vert d \Vert_{X}, G^{\prime}(x) d \in {\cal T}_{K}(G(x)) \}.
\end{align*}
\end{definition}
\noindent
The SOSC is often assumed to analyze the local behavior of optimization methods and ensures the following second-order growth condition at $\bar{x}$: There exist $m > 0$ and $r > 0$ such that
\begin{align*}
f(x) \geq f(\bar{x}) + m \Vert x - \bar{x} \Vert_{X}^{2}
\end{align*}
for any $x \in \Phi \cap B_{X}(\bar{x},r)$. This is one of the important properties in local analysis. Kanzow et al.~\cite{KaSt18} and B\"{o}rgens et al.~\cite{BoKaSt19} discussed the local convergence of the proposed AL methods as stated in Section~\ref{sec:intro}, and they also used the SOSC in their local analyses. Note that the following another SOSC was assumed in Kanzow et al.~\cite{KaSt18} whereas the same one as Definition~\ref{def:SOSC} was used in B\"{o}rgens et al.~\cite{BoKaSt19}: 
\begin{align}
\langle L_{xx}(\bar{v}) d, d \rangle_{X^{\ast}, X} \geq c \Vert d \Vert_{X}^{2} \label{anotherSOSC}
\end{align}
for all $d \in \widetilde{C}_{\eta}(\bar{x})$, where 
\begin{align*}
\widetilde{C}_{\eta}(\bar{x}) \coloneqq \{ d \in X; \langle f^{\prime}(\bar{x}), d \rangle_{X^{\ast}, X} \leq \eta \Vert d \Vert_{X}, \dist (G^{\prime}(\bar{x}) d, {\cal T}_{K}(G(\bar{x})) ) \leq \eta \Vert d \Vert_{X} \}.
\end{align*}
The SOSC given by~\eqref{anotherSOSC} is stronger than that of Definition~\ref{def:SOSC} since $C_{\eta}(\bar{x}) \subset \widetilde{C}_{\eta}(\bar{x})$.


\section{Stabilized SQP method} \label{sec:SSQP}
The former part of this section provides an overview of two types of SQP methods, including ordinary and stabilized versions, and the latter part presents a stabilized SQP method for problem~\eqref{problem:main}. To this end, we consider the following generalized equation~(GE) of the KKT conditions for~\eqref{problem:main}:
\begin{align} \label{GE:KKT}
0 \in
\begin{bmatrix}
L_{x}(v)
\\
-G(x)
\end{bmatrix}
+
\begin{bmatrix}
\{ 0 \}
\\
{\cal N}_{K}^{-1}(\lambda)
\end{bmatrix}
.
\end{align}
Now, we apply the Josephy-Newton method to GE~\eqref{GE:KKT}. Let $v = (x,\lambda) \in V$ be given. The Josephy-Newton method iteratively solves the subproblem below:
\begin{align} \label{GE:subproblem}
0 \in 
\begin{bmatrix}
L_{x}(v)
\\
-G(x)
\end{bmatrix}
+
\begin{bmatrix}
L_{xx}(v) & G^{\prime}(x)^{\ast}
\\
-G^{\prime}(x) & 0
\end{bmatrix}
\begin{bmatrix}
\xi - x
\\
\mu - \lambda
\end{bmatrix}
+
\begin{bmatrix}
\{ 0 \}
\\
{\cal N}_{K}^{-1}(\mu)
\end{bmatrix}
,
\end{align}
which is equivalent to finding a solution $(d, \mu) \in V$ satisfying $f^{\prime}(x) + L_{xx}(v) d + G^{\prime}(x)^{\ast} \mu = 0$ and $\mu \in {\cal N}_{K}( G(x) + G^{\prime}(x) d )$.
Furthermore, they coincide with the KKT conditions of the following problem:
\begin{align*}
\mbox{P$_{0}(v)$}\qquad
\begin{aligned}
& \displaystyle \mini_{d \in X} & & \langle f^{\prime}(x), d \rangle_{X^{\ast}, X} + \frac{1}{2} \langle L_{xx}(v) d, d \rangle_{X^{\ast}, X}
\\
& \subj & & G(x) + G^{\prime}(x) d \in K.
\end{aligned}
\end{align*}
We recall that P$_{0}(v)$ corresponds to the subproblem of the standard SQP methods. Accordingly, the SQP methods are interpreted as the Josephy-Newton methods for GE~\eqref{GE:KKT}. Hence, they obtain a solution $(d, \mu)$ of P$_{0}(v)$ at each iteration and update $x$ and $\lambda$ as $x \leftarrow x + d$ and $\lambda \leftarrow \mu$, respectively.
\par
Meanwhile, the stabilized version of the SQP methods~\cite{FeSo10,GiRo13,Ha99,IzSo12,IzSoUs15,Wr98} has been proposed as a suitable method for degenerate optimization problems satisfying no CQs. In general, the boundedness of the Lagrange multipliers satisfying the KKT conditions is not ensured for degenerate problems, and thus sequences of the Lagrange multipliers generated by iterative methods may diverge. To avoid this phenomenon in the Josephy-Newton method, we consider adding a regularization term regarding the Lagrange multipliers to subproblem~\eqref{GE:subproblem}. Then, one of the methods can be considered to adopt the following subproblem instead of~\eqref{GE:subproblem}:
\begin{align} \label{GE:stabilized_subproblem}
0 \in 
\begin{bmatrix}
L_{x}(v)
\\
-G(x)
\end{bmatrix}
+
\begin{bmatrix}
L_{xx}(v) & G^{\prime}(x)^{\ast}
\\
-G^{\prime}(x) & \sigma I
\end{bmatrix}
\begin{bmatrix}
\xi - x
\\
\mu - \lambda
\end{bmatrix}
+
\begin{bmatrix}
\{ 0 \}
\\
{\cal N}_{K}^{-1}(\mu)
\end{bmatrix}
,
\end{align}
where $\sigma$ is a positive parameter and $I$ denotes the identity mapping on $Y$. Based on the same aforementioned argument, solving~\eqref{GE:stabilized_subproblem} is reduced to finding a KKT point of the following problem:
\begin{align*}
\begin{aligned}
& \displaystyle \mini_{(d, \mu) \in V} & & \langle f^{\prime}(x), d \rangle_{X^{\ast}, X} + \frac{1}{2} \langle L_{xx}(v) d, d \rangle_{X^{\ast}, X} + \frac{\sigma}{2} \Vert \mu \Vert_{Y}^{2}
\\
& \subj & & G(x) + G^{\prime}(x) d - \sigma (\mu - \lambda) \in K.
\end{aligned}
\end{align*}
This is referred to as the stabilized subproblem and always has a feasible solution, such as $(0, \frac{1}{\sigma} G(x) + \lambda - \frac{1}{\sigma} P_{K}(G(x)))$, whereas such feasibility is not ensured for subproblem P$_{0}(v)$ of the ordinary SQP methods. The stabilized SQP methods iteratively solve the stabilized subproblem to obtain its solution $(d, \mu)$ and update the current point $(x, \lambda)$ by the same rule as the ordinary SQP methods. In this study, the positive parameter $\sigma$ is adopted as
\begin{align}
\sigma(v) \coloneqq \Vert f^{\prime}(x) + G^{\prime}(x)^{\ast} \lambda \Vert_{X^{\ast}} + \Vert G(x) - P_{K}(G(x) + \lambda) \Vert_{Y}, \label{def:sigma}
\end{align}
and the proposed stabilized SQP method iteratively solves
\begin{align*}
\mbox{P$(v)$}\qquad
\begin{aligned}
& \displaystyle \mini_{(d, \mu) \in V} & & \langle f^{\prime}(x), d \rangle_{X^{\ast}, X} + \frac{1}{2} \langle L_{xx}(v) d, d \rangle_{X^{\ast}, X} + \frac{\sigma(v)}{2} \Vert \mu \Vert_{Y}^{2}
\\
& \subj & & G(x) + G^{\prime}(x) d - \sigma(v) (\mu - \lambda) \in K.
\end{aligned}
\end{align*}
The formal statement of the proposed method is as follows:
\begin{algorithm}[tbh]
\caption{(The proposed stabilized SQP method)} \label{Local_SSQP}
\begin{algorithmic}[1]
\State{Choose $v_{0} \coloneqq (x_{0}, \lambda_{0}) \in V$ and set $k \coloneqq 0$.} \Comment{Step~0}

\State{Stop if $v_{k}$ satisfies the termination criterion.} \Comment{Step~1}

\State{Solve~P$(v_{k})$ and find its approximate local optimum $(d_{k}, \mu_{k}) \in V$.} \Comment{Step~2}

\State{Calculate $x_{k+1} \coloneqq x_{k} + d_{k}$, $\lambda_{k+1} \coloneqq \mu_{k}$, and $v_{k+1} \coloneqq (x_{k+1}, \lambda_{k+1})$.} \Comment{Step~3}

\State{Update $k \leftarrow k+1$ and return to Step~1.} \Comment{Step~4}
\end{algorithmic}
\end{algorithm}

We recall that Algorithm~\ref{Local_SSQP} is based on the Newton-type method and is designed to achieve the local fast convergence to the KKT point $v$. Consequently, the unit step size is adopted.
\par
In the remainder of this section, we briefly discuss the well-definedness of Algorithm~\ref{Local_SSQP}. Note that this depends on whether problem P$(v)$ can be solved at each iteration point $v$. In general, the solvability of the subproblem is not ensured because there is no guarantee regarding the reflexivity of $X$ and the coerciveness of the mapping $d \mapsto \langle L_{xx}(v) d, d \rangle_{X^{\ast}, X}$. However, if $v$ is sufficiently close to the KKT point $\bar{v}$, then P$(v)$ can be approximately solved under suitable assumptions. This fact is shown in the next section.

\section{Local convergence analysis} \label{sec:local_convergence}
This section provides local convergence analysis of Algorithm~\ref{Local_SSQP}. Throughout the section, we denote by $\Gamma$ the set of KKT points for problem~\eqref{problem:main}, and assume that there exists at least one KKT point $\bar{v} \coloneqq (\bar{x}, \bar{\lambda}) \in \Gamma$. Moreover, we make the following assumptions:
\begin{description}
\item[{\rm (A1)}] The SRCQ holds at $\bar{v}$;
\item[{\rm (A2)}] the SOSC holds at $\bar{v}$;
\item[{\rm (A3)}] the second derivatives of $f$ and $G$ are locally Lipschitz continuous at $\bar{x}$.
\end{description}
We also suppose that Algorithm~\ref{Local_SSQP} generates an infinite sequence $\{ v_{k} \} \subset V$ such that $v_{k} \not \in \Gamma$ for each $k \in \mathbb{N} \cup \{ 0 \}$.
\par
From now on, we prepare some inequalities for the subsequent local analysis. Let us define $\Omega \coloneqq \{ d \in X; G^{\prime}(\bar{x}) d \in {\cal T}_{K}(G(\bar{x})) \}$. Thanks to (A1), the RCQ is valid, namely, $0 \in {\rm int}(G(\bar{x}) + G^{\prime}(\bar{x}) X - K)$. Meanwhile, we readily have ${\rm int}(G(\bar{x}) + G^{\prime}(\bar{x})X - K) \subset {\rm int}(G^{\prime}(\bar{x})X - {\cal T}_{K}(G(\bar{x})))$ because $K - G(\bar{x}) \subset {\cal T}_{K}(G(\bar{x}))$ is satisfied. These facts lead to $0 \in {\rm int}(G^{\prime}(\bar{x}) X - {\cal T}_{K}(G(\bar{x})))$. Hence, by applying~\cite[Theorem~2.87]{BoSh00}, there exists $\kappa > 0$ and $\nu_{1} > 0$ such that $\dist(d, \Omega) \leq \kappa \dist(G^{\prime}(\bar{x}) d, {\cal T}_{K}(G(\bar{x})))$ for any $d \in \nu_{1} B_{X}$. Noting $K - G(\bar{x}) \subset {\cal T}_{K}(G(\bar{x}))$ implies $\dist(G^{\prime}(\bar{x}) d, {\cal T}_{K}(G(\bar{x})) ) \leq \dist(G(\bar{x}) + G^{\prime}(\bar{x}) d, K)$, and hence for each $d \in \nu_{1} B_{X}$,
\begin{align}
\dist(d, \Omega) \leq \kappa \dist(G(\bar{x}) + G^{\prime}(\bar{x}) d, K). \label{ineq:dist_dPhi}
\end{align}

\noindent
Let ${\cal F}$ be define as the objective function of P$_{0}(\bar{v})$, namely, ${\cal F}(d) \coloneqq \langle f^{\prime}(\bar{x}), d \rangle_{X^{\ast}, X} + \frac{1}{2} \langle L_{xx}(\bar{v}) d, d \rangle_{X^{\ast}, X}$. From the differentiability of ${\cal F}$ at $d = 0$, there exists $\nu_{2} > 0$ such that $| {\cal F}(d) -{\cal F}(0) - \langle {\cal F}^{\prime}(0), d \rangle_{X^{\ast}, X} | \leq \frac{\eta}{8} \Vert d \Vert_{X}$ for all $d \in \nu_{2} B_{X}$. Then, utilizing ${\cal F}(d) -{\cal F}(0) - \langle {\cal F}^{\prime}(0), d \rangle_{X^{\ast}, X} = \frac{1}{2} \langle L_{xx}(\bar{v}) d, d \rangle_{X^{\ast}, X}$ yields that for each $d \in \nu_{2} B_{X}$,
\begin{align}
\langle L_{xx}(\bar{v}) d, d \rangle_{X^{\ast}, X} \geq - \frac{\eta}{4} \Vert d \Vert_{X}.  \label{ineq:calF_Lxx}
\end{align}

\par
Let $v = (x, \lambda) \in V \backslash \Gamma$ be given. In what follows, we discuss the solvability of problem~P$(v)$. Since there is no guarantee that the subproblem can be solved as stated in the previous section, let us consider the following problem instead of~P$(v)$:
\begin{align*}
\mbox{Q$(v)$}\qquad
\begin{aligned}
& \displaystyle \mini_{(d, \mu) \in V} & & \langle f^{\prime}(x), d \rangle_{X^{\ast}, X} + \frac{1}{2} \langle L_{xx}(v) d, d \rangle_{X^{\ast}, X} + \frac{\sigma(v)}{2} \Vert \mu \Vert_{Y}^{2}
\\
& \subj & & G(x) + G^{\prime}(x) d - \sigma(v) (\mu - \lambda) \in K, \quad \Vert d \Vert_{X} \leq \nu,
\end{aligned}
\end{align*}
where $\nu \coloneqq \min \{ \nu_{1}, \nu_{2} \}$. Notice that Q$(v)$ is an optimization problem obtained by adding the constraint $\Vert d \Vert_{X} \leq \nu$ to P$(v)$. As shown in Theorem~\ref{theorem:omega_sigma}, approximate solutions of Q$(v)$ are also those of P$(v)$ under (A1) and (A2) if $v$ is close to $\bar{v}$. Hence, using problem~Q$(v)$ is reasonable when $x$ is close to $\bar{x}$. Such a {\it localized technique} is often used in the local convergence analysis. For example, see~\cite{BoKaSt19,FeSo12,Ro82,SuSuZh08}. The subsequent local analysis also uses the localized technique. To begin with, we show that problem~Q$(v)$ can be approximately solved thanks to the spherical constraint.
\begin{theorem} \label{theorem:eps_optimal}
Let $v \in V \backslash \Gamma$ and $\varepsilon > 0$. Then, problem ${\rm Q}(v)$ has an $\varepsilon$-optimal solution $w(v) = (d(v), \mu(v)) \in V$. Moreover, there exist another $\varepsilon$-optimal solution $\widehat{w}(v) = (\widehat{d}(v), \widehat{\mu}(v)) \in V$ and $\rho(v) \in Y$ such that
\begin{align*}
\begin{gathered}
\Vert w(v) - \widehat{w}(v) \Vert_{V} \leq \sqrt{\varepsilon},
~ \dist(- f^{\prime}(x) - L_{xx}(v) \widehat{d}(v) - G^{\prime}(x)^{\ast} \rho(v), {\cal N}_{\nu B_{X}}(\widehat{d}(v))) \leq \sqrt{\varepsilon},
\\
\sigma(v) \Vert \rho(v) - \widehat{\mu}(v) \Vert_{Y} \leq \sqrt{\varepsilon},
~ \rho(v) \in {\cal N}_{K}( G(x) + G^{\prime}(x) \widehat{d}(v) - \sigma(v)(\widehat{\mu}(v) - \lambda) ).
\end{gathered}
\end{align*}
\end{theorem}

\begin{proof}
Let $v = (x, \lambda) \in V \backslash \Gamma$ be a fixed point, and let $\varphi \colon V \to \mathbb{R}$, $\Psi \colon V \to Y$, and ${\cal S} \subset V$ be defined by $\varphi(d, \mu) \coloneqq \langle f^{\prime}(x), d \rangle_{X^{\ast}, X} + \frac{1}{2} \langle L_{xx}(v) d, d \rangle_{X^{\ast}, X} + \frac{\sigma(v)}{2} \Vert \mu \Vert_{Y}^{2}$, $\Psi(d,\mu) \coloneqq G(x) + G^{\prime}(x) d - \sigma(v) (\mu - \lambda)$, and ${\cal S} \coloneqq \{ (d, \mu) \in V; \Psi(d, \mu) \in K, \, \Vert d \Vert_{X} \leq \nu \}$, respectively. Note that $\sigma(v) > 0$ holds from $v \not \in \Gamma$ and that ${\cal S}$ is nonempty because $(0, \frac{1}{\sigma(v)}G(x) + \lambda - \frac{1}{\sigma(v)} P_{K}(G(x) + \sigma(v) \lambda) )$ is a feasible solution. 
\par
Now, we show that Q$(v)$ has an $\varepsilon$-optimal solution. For each $(d, \mu) \in {\cal S}$, we have $\varphi(d, \mu) \geq - \nu \Vert f^{\prime}(x) \Vert_{X^{\ast}} - \frac{\nu^{2}}{2} \Vert L_{xx}(v) \Vert_{X \to X^{\ast}}$, that is, $\inf \{ \varphi(d,\mu); (d,\mu) \in {\cal S} \} > -\infty$. Thus, there exists an $\varepsilon$-optimal solution $(d(v), \mu(v)) \in {\cal S}$ of Q$(v)$.
\par
Next, we show the latter part of the assertion. Let $(d, \mu) \in {\cal S}$ be arbitrary. Note that $\Psi^{\prime}(d, \mu)[\, \cdot \, , \, \cdot \, ] = G^{\prime}(x)[\, \cdot \,] - \sigma(v)[\, \cdot \,]$ is a surjective mapping from $V$ to $Y$ because $\sigma(v) > 0$ holds. Indeed, for any $y \in Y$, there exists $(0, - \frac{1}{\sigma(v)} y) \in V$ such that $\Psi^{\prime}(d, \mu)[0, - \frac{1}{\sigma(v)} y] = G^{\prime}(x) [0] - \sigma(v) [- \frac{1}{\sigma(v)} y] = y$. Hence, we obtain $\Psi(d, \mu) + \Psi^{\prime}(d, \mu) [ \nu B_{X} - d, V - \mu ] - K = Y$, i.e.,
\begin{align*}
0 \in Y = {\rm int} (Y) = {\rm int}( \Psi(d, \mu) + \Psi^{\prime}(d, \mu) [ \nu B_{X} - d, V - \mu ] - K ).
\end{align*}
This result suggests that problem~Q$(v)$ satisfies the RCQ at any $(d, \mu) \in {\cal S}$. Then, the assertion of this theorem follows from~\cite[Theorem~3.23]{BoSh00}.
\end{proof}

Theorem~\ref{theorem:eps_optimal} states that problem~Q$(v)$ has an $\varepsilon$-optimal solution whereas the existence of exact optimal solutions is not ensured. 
\par
The goal of the former part of Section~\ref{sec:local_convergence} is to show that problem~P$(v)$ has an approximate solution. Accordingly, we prepare some properties regarding the function~$\sigma$ defined by~\eqref{def:sigma}, such as the error bound property and the Lipschitz continuity.
\begin{proposition} \label{prop:error_bound}
If {\rm (A1)} and {\rm (A2)} hold, then there exist $m > 0$ and $r > 0$ such that $\Vert v - \bar{v} \Vert_{V} \leq m \sigma(v)$ for all $v \in B_{V}(\bar{v}, r)$.
\end{proposition}

\begin{proof}
Since (A1) and (A2) hold, we can utilize~\cite[Theorem~3.2]{KaSt18}. Therefore, there exist $m > 0$ and $\gamma > 0$ such that
\begin{align}
\Vert v - \bar{v} \Vert_{V} \leq m \sigma(v) \label{ineq:vv_sigma}
\end{align}
for all $v \in \{ v = (x, \lambda) \in V;  \Vert x - \bar{x} \Vert_{X} \leq \gamma,  \sigma(v) \leq \gamma \}$. Note that $\Vert x - \bar{x} \Vert_{X} \to 0$ and $\sigma(v) \to 0$ as $v \to \bar{v}$, where $v = (x, \lambda) \in V$. Thus, there exists $r > 0$ such that $\Vert x - \bar{x} \Vert_{X} \leq \gamma$ and $\sigma(v) \leq \gamma$ for all $v \in B_{V}(\bar{v}, r)$. It then follows from~\eqref{ineq:vv_sigma} that $\Vert v - \bar{v} \Vert_{V} \leq m \sigma(v)$ for all $v \in B_{V}(\bar{v}, r)$. Therefore, the proof is completed.
\end{proof}

\begin{proposition} \label{prop:sigma_Lipschitz}
There exist $m > 0$ and $r > 0$ such that $\sigma(v) \leq m \Vert v - \bar{v} \Vert_{V}$ for all $v \in B_{V}(\bar{v}, r)$.
\end{proposition}

\begin{proof}
We first recall that $f$ and $G$ are twice continuously differentiable on $X$. These facts show that $L_{x}$ and $G$ are locally Lipschitz continuous at $\bar{v}$. Then, there exist $b > 0$ and $r > 0$ such that
\begin{align}
\Vert L_{x}(v) - L_{x} (\bar{v}) \Vert_{X^{\ast}} \leq b \Vert v - \bar{v} \Vert_{V}, \quad \Vert G(x) - G(\bar{x}) \Vert_{Y} \leq b \Vert x - \bar{x} \Vert_{X} \label{ineq:LLip}
\end{align}
for all $v \in B_{V}(\bar{v}, r)$. Let $m \coloneqq 3 b + 1$, and let $v = (x, \lambda) \in B_{V}(\bar{v}, r)$ be arbitrary. Notice that $L_{x}(\bar{v}) = 0$ and $G(\bar{x}) = P_{K}(G(\bar{x}) + \bar{\lambda})$. Then, using~\eqref{ineq:LLip} and the nonexpansiveness of $P_{K}$ yields
\begin{align*}
\sigma(v)
&= \Vert L_{x}(v) - L_{x}(\bar{v}) \Vert_{X^{\ast}} + \Vert G(x) - P_{K}(G(x) + \lambda) - G(\bar{x}) + P_{K}(G(\bar{x}) + \bar{\lambda}) \Vert_{Y}
\\
&\leq b \Vert v - \bar{v} \Vert_{V} + \Vert G(x) - G(\bar{x}) \Vert_{Y} + \Vert P_{K}(G(x) + \lambda) - P_{K}(G(\bar{x}) + \bar{\lambda}) \Vert_{Y}
\\
&\leq b \Vert v - \bar{v} \Vert_{V} + 2 \Vert G(x) - G(\bar{x}) \Vert_{Y} + \Vert \lambda - \bar{\lambda} \Vert_{Y}
\\
&\leq b \Vert v - \bar{v} \Vert_{V} + 2 b \Vert x - \bar{x} \Vert_{X} + \Vert \lambda - \bar{\lambda} \Vert_{Y}
\\
&\leq m \Vert v - \bar{v} \Vert_{V},
\end{align*}
that is, the desired result is obtained.
\end{proof}

\begin{lemma} \label{lemma:three_results}
Suppose that {\rm (A1)} and {\rm (A2)} are satisfied. For each $v \in V \backslash \Gamma$, let $w(v) = (d(v), \mu(v)) \in V$ be a $\sigma(v)^{6}$-optimal solution of ${\rm Q}(v)$. Then, the following statements hold:
\begin{itemize}
\item[{\rm (i)}] $\sigma(v) \Vert \mu(v) \Vert_{Y} \to 0$ and $\dist(G(\bar{x}) + G^{\prime}(\bar{x}) d(v), K) \to 0$ as $\Vert v - \bar{v} \Vert_{V} \to 0$;

\item[{\rm (ii)}] there exist $m > 0$ and $r > 0$ such that $\langle f^{\prime}(x), d(v) \rangle_{X^{\ast}, X} + \frac{1}{2} \langle L_{xx}(v) d(v), d(v) \rangle_{X^{\ast}, X} \leq m \sigma(v)$ for all $v = (x, \lambda) \in B_{V}(\bar{v}, r) \backslash \Gamma$.
\end{itemize}
\end{lemma}

\begin{proof}
Since (A1) and (A2) hold, Proposition~\ref{prop:error_bound} ensures that there exist $m_{1} > 0$ and $r_{1} > 0$ such that for all $v \in B_{V}(\bar{v}, r_{1})$,
\begin{align}
\Vert v - \bar{v} \Vert_{V} \leq m_{1} \sigma(v). \label{ineq:error_bound}
\end{align}
By the differentiability of $G$, there exists $r_{2} \in (0, r_{1})$ such that for all $x \in B_{X}(\bar{x}, r_{2})$, $\Vert G(x) - G(\bar{x}) - G^{\prime}(\bar{x}) (x - \bar{x}) \Vert_{Y} \leq \Vert x - \bar{x} \Vert_{X}$. It then follows from~\eqref{ineq:error_bound} that for all $v \in B_{V}(\bar{v}, r_{2})$,
\begin{align}
\Vert G(x) - G(\bar{x}) \Vert_{Y} \leq m_{2} \sigma(v), \label{ineq:G_sigma}
\end{align}
where $m_{2} \coloneqq m_{1} (1 + \Vert G^{\prime}(\bar{x}) \Vert_{X \to Y})$. Let $v = (x, \lambda) \in V \backslash \Gamma$ be arbitrary. We note that $\sigma(v) > 0$ holds, $(0, \frac{1}{\sigma(v)}G(x) + \lambda - \frac{1}{\sigma(v)} P_{K}(G(x)) )$ is a feasible solution of Q$(v)$, and $w(v) = (d(v), \mu(v)) \in V$ is a $\sigma(v)^{6}$-optimal solution of Q$(v)$. These facts derive
\begin{align}
&\langle f^{\prime}(x), d(v) \rangle_{X^{\ast}, X} + \frac{1}{2} \langle L_{xx}(v) d(v), d(v) \rangle_{X^{\ast}, X} + \frac{\sigma(v)}{2} \Vert \mu(v) \Vert_{Y}^{2} \nonumber
\\
& < \frac{1}{2 \sigma(v)} \left\Vert G(x) + \sigma(v) \lambda - P_{K}(G(x)) \right\Vert_{Y}^{2} + \sigma(v)^{6}. \label{ineq:e-optimal}
\end{align}
Using~\eqref{ineq:e-optimal} and $d(v) \in \nu B_{X}$ yields
\begin{align}
\sigma(v)^{2} \Vert \mu(v) \Vert_{Y}^{2} 
& \leq 2 \sigma(v) \left( \nu \Vert f^{\prime}(x) \Vert_{X^{\ast}} + \frac{\nu^{2}}{2} \Vert L_{xx}(v) \Vert_{X \to X^{\ast}} \right)  + \Vert G(x) + \sigma(v) \lambda - P_{K}(G(x)) \Vert_{Y}^{2} + 2 \sigma(v)^{7}. \label{ineq:sigma_zeta}
\end{align}
Now, we recall that $\sigma(v) \to 0$ and $G(x) + \sigma(v) \lambda - P_{K}(G(x)) \to G(\bar{x}) - P_{K}(G(\bar{x})) = 0$ if $\Vert v - \bar{v} \Vert_{V} \to 0$. Thus, inequality~\eqref{ineq:sigma_zeta} implies $\sigma(v) \Vert \mu(v) \Vert_{Y} \to 0$ as $\Vert v - \bar{v} \Vert_{V} \to 0$.
\par
From $\dist( G(x) + G^{\prime}(x) d(v) - \sigma(v) (\mu(v) - \lambda), K ) = 0$, $d(v) \in \nu B_{X}$, the nonexpansiveness of the distance function, and the continuity of $G$ and $G^{\prime}$, we obtain
\begin{align*}
 |\dist( G(\bar{x}) + G^{\prime}(\bar{x}) d(v), K )|
&= |\dist( G(\bar{x}) + G^{\prime}(\bar{x}) d(v), K ) - \dist( G(x) + G^{\prime}(x) d(v) - \sigma(v) (\mu(v) - \lambda), K )|
\\
&\leq \Vert G(x) - G(\bar{x}) \Vert_{Y} + \nu \Vert G^{\prime}(x) - G^{\prime}(\bar{x}) \Vert_{X \to Y} + \sigma(v) \Vert \mu(v) \Vert_{Y} + \sigma(v) \Vert \lambda \Vert_{Y}.
\end{align*}
The inequality and the continuity of $G$ and $G^{\prime}$ derive $\dist( G(\bar{x}) + G^{\prime}(\bar{x}) d(v), K ) \to 0$ as $\Vert v - \bar{v} \Vert_{V} \to 0$.
\par
We define $m \coloneqq 1 + \frac{1}{2} (\Vert \bar{\lambda} \Vert_{Y} + r_{2} + 2 m_{2})^{2}$. Since $\sigma(v) \to 0$ as $\Vert v - \bar{v} \Vert_{V} \to 0$, there exists $r \in (0, r_{2})$ such that $\sigma(v) \leq 1$ for all $v \in B_{V}(\bar{v}, r)$. Let $v = (x, \lambda) \in B_{V}(\bar{v}, r) \backslash \Gamma$ be arbitrary. Then, it is clear that $v \in B_{V}(\bar{v}, r_{2}) \backslash \Gamma$, and hence inequality~\eqref{ineq:G_sigma} is verified. It then follows from $G(\bar{x}) = P_{K}(G(\bar{x}))$ and the nonexpansiveness of $P_{K}$ that
\begin{align}
 \left\Vert G(x) + \sigma(v) \lambda - P_{K}(G(x)) \right\Vert_{Y}
& \leq \sigma(v) \Vert \lambda \Vert_{Y} + \Vert G(x) - G(\bar{x}) \Vert_{Y} + \Vert P_{K}(G(x)) - P_{K}(G(\bar{x})) \Vert_{Y} \nonumber
\\
&\leq \sigma(v) ( \Vert \lambda \Vert_{Y} + 2 m_{2} ) \nonumber
\\
&\leq \sigma(v) ( \Vert \bar{\lambda} \Vert_{Y} + r_{2} + 2m_{2} ), \label{ineq:G_PK_sigma}
\end{align}
where the last inequality is derived from $\Vert \lambda \Vert_{Y} \leq \Vert \bar{\lambda} \Vert_{Y} + \Vert \lambda - \bar{\lambda} \Vert_{Y} \leq \Vert \bar{\lambda} \Vert_{Y} + \Vert v - \bar{v} \Vert_{V} \leq \Vert \bar{\lambda} \Vert_{Y} + r \leq \Vert \bar{\lambda} \Vert_{Y} + r_{2}$.
Since $0 \leq \frac{\sigma(v)}{2} \Vert \mu \Vert_{Y}^{2}$ and $\sigma(v)^{5} \leq 1$ hold, combining~\eqref{ineq:e-optimal} and~\eqref{ineq:G_PK_sigma} yields $\langle f^{\prime}(x), d(v) \rangle_{X^{\ast}, X} + \frac{1}{2} \langle L_{xx}(v) d(v), d(v) \rangle_{X^{\ast}, X} < \frac{1}{2} (\Vert \bar{\lambda} \Vert_{Y} + r_{2} + 2m_{2})^{2} \sigma(v) + \sigma(v)^{6} \leq m \sigma(v)$.
\end{proof}

\noindent
Since problem~Q$(v)$ admits a $\sigma(v)^{6}$-optimal solution $w(v) = (d(v), \mu(v))$, the key to the existence of approximate solutions of~P$(v)$ depends on whether $\Vert d(v) \Vert_{X} < \nu$ holds or not. The next proposition shows that $\Vert d(v) \Vert_{X} < \nu$ is ensured for any $v$ sufficiently close to $\bar{v}$ and brings some important properties. Hereafter, the notation $\bar{w} \coloneqq (0, \bar{\lambda}) \in V$ is used in the subsequent argument.

\begin{proposition} \label{proposition:xi_zero}
Suppose that {\rm (A1)} and {\rm (A2)} hold. For any $v = (x, \lambda) \in V \backslash \Gamma$, let $w(v) = (d(v), \mu(v)) \in V$ denote a $\sigma(v)^{6}$-optimal solution of ${\rm Q}(v)$. Then, the following statements hold:
\begin{itemize}
\item[{\rm (i)}] If $\Vert v - \bar{v} \Vert_{V} \to 0$, then $\Vert w(v) - \bar{w} \Vert_{V} \to 0$;

\item[{\rm (ii)}] there exists $\gamma > 0$ such that for any $v \in B_{V}(\bar{v}, \gamma) \backslash \Gamma$, problem~${\rm Q}(v)$ has another $\sigma(v)^{6}$-optimal solution $\widehat{w}(v) = (\widehat{d}(v), \widehat{\mu}(v)) \in V$, and there exists $\rho(v) \in Y$ satisfying
\begin{align}
& \Vert w(v) - \widehat{w}(v) \Vert_{V} \leq \sigma(v)^{3}, \label{ineq:sigma4_1}
\\
& \Vert f^{\prime}(x) + L_{xx}(v) \widehat{d}(v) + G^{\prime}(x)^{\ast} \rho(v) \Vert_{X^{\ast}} \leq \sigma(v)^{3}, \label{ineq:sigma4_2}
\\
& \Vert \rho(v) - \widehat{\mu}(v) \Vert_{Y} \leq \sigma(v)^{2}, \label{ineq:sigma4_3}
\\
& \rho(v) \in {\cal N}_{K}( G(x) + G^{\prime}(x) \widehat{d}(v) - \sigma(v) (\widehat{\mu}(v) - \lambda) ). \label{ineq:sigma4_4}
\end{align}
\end{itemize}
\end{proposition}

\begin{proof}
Let us verify that $\Vert d(v) \Vert_{X} \to 0$ if $\Vert v - \bar{v} \Vert_{V} \to 0$. We prove this assertion by contradiction. Then, there exist $\varepsilon \in (0, \nu)$ and $\{ v_{j} \coloneqq (x_{j}, \lambda_{j}) \} \subset V \backslash \Gamma$ such that
\begin{align}
\Vert v_{j} - \bar{v} \Vert_{V} \leq \frac{\nu}{j}, \quad \varepsilon \leq \Vert d(v_{j}) \Vert_{X} \leq \nu \label{ineq:vjxij}
\end{align}
for $j \in \mathbb{N}$. We denote $j \in \mathbb{N}$, $\sigma_{j} \coloneqq \sigma(v_{j})$, $w_{j} \coloneqq w(v_{j})$, $d_{j} \coloneqq d(v_{j})$, and $\mu_{j} \coloneqq \mu(v_{j})$. Let $m_{1} > 0$ and $r_{1} > 0$ be constants described in item~(ii) of Lemma~\ref{lemma:three_results}. We have from~\eqref{ineq:vjxij} that $\Vert v_{j} - \bar{v} \Vert_{V} \to 0$ as $j \to \infty$, namely, there exists $n_{1} \in \mathbb{N}$ such that $v_{j} \in B_{V}(\bar{v}, r_{1}) \backslash \Gamma$ for $j \geq n_{1}$. Hence, item~(ii) of Lemma~\ref{lemma:three_results} yields $\langle f^{\prime}(x_{j}), d_{j} \rangle_{X^{\ast}, X} + \frac{1}{2} \langle L_{xx}(v_{j}) d_{j}, d_{j} \rangle_{X^{\ast},X} \leq m_{1} \sigma_{j}$ for $j \geq n_{1}$. Then, taking $\limsup$ implies
\begin{align}
\limsup_{j \to \infty} \left[ \langle f^{\prime}(x_{j}), d_{j} \rangle_{X^{\ast}, X} + \frac{1}{2} \langle L_{xx}(v_{j}) d_{j}, d_{j} \rangle_{X^{\ast},X} \right] \leq 0. \label{ineq:limsup_fL}
\end{align}
Since $\Vert d_{j} \Vert_{X} \leq \nu$ holds by~\eqref{ineq:vjxij}, we obtain
\begin{align}
& \langle f^{\prime}(x_{j}), d_{j} \rangle_{X^{\ast}, X} + \frac{1}{2} \langle L_{xx}(v_{j}) d_{j}, d_{j} \rangle_{X^{\ast},X} \nonumber
\\
& = \langle f^{\prime}(\bar{x}), d_{j} \rangle_{X^{\ast}, X} + \frac{1}{2} \langle L_{xx} (\bar{v}) d_{j}, d_{j} \rangle_{X^{\ast}, X} + \langle f^{\prime}(x_{j}) - f^{\prime}(\bar{x}), d_{j} \rangle_{X^{\ast}, X} + \frac{1}{2} \langle [ L_{xx} (v) - L_{xx} (\bar{v}) ] d_{j}, d_{j} \rangle_{X^{\ast}, X} \nonumber
\\
& \geq \langle f^{\prime}(\bar{x}), d_{j} \rangle_{X^{\ast}, X} + \frac{1}{2} \langle L_{xx} (\bar{v}) d_{j}, d_{j} \rangle_{X^{\ast}, X} - \nu \Vert f^{\prime}(x_{j}) - f^{\prime}(\bar{x}) \Vert_{X^{\ast}} - \frac{\nu^{2}}{2} \Vert L_{xx}(v_{j}) - L_{xx}(\bar{v}) \Vert_{X \to X^{\ast}}. \label{ineq:fLnu}
\end{align} 
Now, we recall that \eqref{ineq:dist_dPhi} and~\eqref{ineq:calF_Lxx} hold for $d = d_{j}$ thanks to $d_{j} \in \nu B_{X}$.
Let ${\cal I} \coloneqq \{ i \in \mathbb{N}; \langle f^{\prime}(\bar{x}), d_{i} \rangle_{X^{\ast}, X} > \frac{\eta}{4} \Vert d_{i} \Vert_{X} \}$ and ${\cal J} \coloneqq \mathbb{N} \backslash {\cal I}$. In what follows, we consider two cases: (a) $\mbox{card}({\cal I}) < \infty$; (b) $\mbox{card}({\cal I}) = \infty$.
\begin{itemize}
\item[(a)] In this case, there exists $n_{2} \in \mathbb{N}$ such that $j \in {\cal J}$ for $j \geq n_{2}$. It follows from~\eqref{ineq:vjxij} and item~(i) of Lemma~\ref{lemma:three_results} that 
\begin{align}
\tau_{j} \coloneqq \dist(G(\bar{x}) + G^{\prime}(\bar{x}) d_{j}, K) \to 0 \quad (j \to \infty). \label{def:alpha}
\end{align}
From~\eqref{ineq:dist_dPhi},~\eqref{def:alpha}, and the definitions of $\Omega$ and $\dist(d_{j}, \Omega)$, there exists $h_{j} \in X$ such that
\begin{align}
G^{\prime}(\bar{x}) h_{j} \in T_{K}(G(\bar{x})), ~ \Vert d_{j} - h_{j} \Vert_{X} < \dist(d_{j}, \Omega) + \frac{1}{j} \leq \kappa \tau_{j} + \frac{1}{j} \to 0 \quad (j \to \infty). \label{ineq:dh_j_0}
\end{align}
Since $\Vert d_{j} - h_{j} \Vert_{X} \to 0$ as $j \to \infty$, there exists $n_{3} \in \mathbb{N}$ such that for $j \geq n_{3}$,
\begin{align}
\Vert d_{j} - h_{j} \Vert_{X} \leq \min \left\{ \frac{\varepsilon}{2} , \, \frac{\eta \varepsilon}{4(1 + \Vert f^{\prime}(\bar{x}) \Vert_{X^{\ast}})} \right\}. \label{ineq:tau_dh}
\end{align}
Let $j \geq \max \{ n_{2}, n_{3} \}$ be an arbitrary integer. From now on, we show $h_{j} \in C_{\eta}(\bar{x})$. It is sufficient to verify $\langle f^{\prime}(\bar{x}), h_{j} \rangle_{X^{\ast}, X} \leq \eta \Vert h_{j} \Vert_{X}$ thanks to the former result of~\eqref{ineq:dh_j_0}. Since $\Vert d_{j} - h_{j}\Vert_{X} \leq \frac{\varepsilon}{2} \leq \frac{1}{2} \Vert d_{j} \Vert_{X}$ holds by~\eqref{ineq:vjxij}~and~\eqref{ineq:tau_dh}, we obtain
\begin{align}
& \frac{1}{2} \Vert d_{j} \Vert_{X} = \Vert d_{j} \Vert_{X} - \frac{1}{2} \Vert d_{j} \Vert_{X} \leq \Vert d_{j} \Vert_{X} - \Vert d_{j} - h_{j} \Vert_{X} \leq \Vert h_{j} \Vert_{X}, \label{ineq:dh1}
\\
& \Vert h_{j} \Vert_{X} \leq \Vert d_{j} - h_{j} \Vert_{X} + \Vert d_{j} \Vert_{X} \leq \frac{1}{2} \Vert d_{j} \Vert_{X} + \Vert d_{j} \Vert_{X} = \frac{3}{2} \Vert d_{j} \Vert_{X}. \label{ineq:dh2}
\end{align}
We recall that $\langle f^{\prime}(\bar{x}), d_{j} \rangle_{X^{\ast}, X} \leq \frac{\eta}{4} \Vert d_{j} \Vert_{X}$ by $j \in {\cal J}$. It then follows from \eqref{ineq:vjxij},~\eqref{ineq:tau_dh}, and~\eqref{ineq:dh1} that $\langle f^{\prime}(\bar{x}), h_{j} \rangle_{X^{\ast}, X} \leq \langle f^{\prime}(\bar{x}), d_{j} \rangle_{X^{\ast}, X} + \Vert f^{\prime}(\bar{x}) \Vert_{X^{\ast}} \Vert d_{j} - h_{j} \Vert_{X} \leq \frac{\eta}{4} \Vert d_{j} \Vert_{X} + \frac{\eta}{4} \Vert d_{j} \Vert_{X} \leq \eta \Vert h_{j} \Vert_{X}$, that is, $h_{j} \in C_{\eta}(\bar{x})$. Hence, combining~(A2) and~\eqref{ineq:dh1} implies
\begin{align}
\langle L_{xx}(\bar{v}) h_{j}, h_{j} \rangle_{X^{\ast}, X} \geq c \Vert h_{j} \Vert_{X}^{2} \geq \frac{c}{4} \Vert d_{j} \Vert_{X}^{2}. \label{ineq:Lxxc}
\end{align}
We have from $\langle f^{\prime}(\bar{x}), d_{j} \rangle_{X^{\ast}, X} = - \langle G^{\prime}(\bar{x})^{\ast} \bar{\lambda}, d_{j} \rangle_{X^{\ast}, X}$, \eqref{ineq:vjxij}, \eqref{ineq:dh2}, and~\eqref{ineq:Lxxc} that 
\begin{align}
& \langle f^{\prime}(\bar{x}), d_{j} \rangle_{X^{\ast}, X} + \frac{1}{2} \langle L_{xx} (\bar{v}) d_{j}, d_{j} \rangle_{X^{\ast}, X} \nonumber
\\
& = \frac{1}{2} \langle L_{xx} (\bar{v}) h_{j}, h_{j} \rangle_{X^{\ast}, X} + \frac{1}{2} \langle L_{xx} (\bar{v}) d_{j}, d_{j} - h_{j} \rangle_{X^{\ast}, X} + \frac{1}{2} \langle L_{xx} (\bar{v}) (d_{j} - h_{j}), h_{j} \rangle_{X^{\ast}, X} - \langle G^{\prime}(\bar{x})^{\ast} \bar{\lambda}, d_{j} \rangle_{X^{\ast}, X} \nonumber
\\
& \geq \frac{c \varepsilon^{2}}{8} - \frac{5\nu}{4} \Vert L_{xx} (\bar{v}) \Vert_{X \to X^{\ast}} \Vert d_{j} - h_{j} \Vert_{X} - ( \bar{\lambda}, G^{\prime}(\bar{x}) d_{j} )_{Y}. \label{ineq:Lxx}
\end{align}
On the other hand, the definition of $\tau_{j}$ guarantees the existence of $p_{j} \in K$ satisfying $\Vert G(\bar{x}) + G^{\prime}(\bar{x}) d_{j} - p_{j} \Vert_{Y} < \tau_{j} + \frac{1}{j}$. Moreover, we obtain $(\bar{\lambda}, G(\bar{x}) - p_{j} )_{Y} \geq 0$ thanks to $\bar{\lambda} \in {\cal N}_{K}(G(\bar{x}))$. These facts derive $- (\bar{\lambda}, G^{\prime}(\bar{x}) d_{j})_{Y} = - (\bar{\lambda}, G(\bar{x}) + G^{\prime}(\bar{x}) d_{j} - p_{j})_{Y} + (\bar{\lambda}, G(\bar{x}) - p_{j} )_{Y} \geq - \Vert \bar{\lambda} \Vert_{Y} ( \tau_{j} + \frac{1}{j} )$.
Then, we have by~\eqref{ineq:fLnu} and~\eqref{ineq:Lxx} that
\begin{align*}
\langle f^{\prime}(x_{j}), d_{j} \rangle_{X^{\ast}, X} + \frac{1}{2} \langle L_{xx}(v_{j}) d_{j}, d_{j} \rangle_{X^{\ast},X}
& \geq \frac{c\varepsilon^{2}}{8} - \frac{5\nu}{4} \Vert L_{xx} (\bar{v}) \Vert_{X \to X^{\ast}} \Vert d_{j} - h_{j} \Vert_{X} - \Vert \bar{\lambda} \Vert_{Y} \left( \tau_{j} + \frac{1}{j} \right)  
\\
& ~ \quad - \nu \Vert f^{\prime}(x_{j}) - f^{\prime}(\bar{x}) \Vert_{X^{\ast}} - \frac{\nu^{2}}{2} \Vert L_{xx}(v_{j}) - L_{xx}(\bar{v}) \Vert_{X \to X^{\ast}}.
\end{align*}
Thus, using~\eqref{ineq:vjxij},~\eqref{def:alpha}, and~\eqref{ineq:dh_j_0} implies
\begin{align*}
\limsup_{j \to \infty} \left[ \langle f^{\prime}(x_{j}), d_{j} \rangle_{X^{\ast}, X} + \frac{1}{2} \langle L_{xx}(v_{j}) d_{j}, d_{j} \rangle_{X^{\ast},X} \right] \geq \frac{c\varepsilon^{2}}{8} > 0.
\end{align*}
\item[(b)] This case means that there exists an infinite sequence $\{ j_{k} \} \subset \mathbb{N}$ satisfying $j_{k} \in {\cal I}$ for any $k \in \mathbb{N}$. Let $k \in \mathbb{N}$ be arbitrary. It follows from $\langle f^{\prime}(\bar{x}), d_{j_{k}} \rangle_{X^{\ast},X} > \frac{\eta}{4} \Vert d_{j_{k}} \Vert_{X}$,~\eqref{ineq:calF_Lxx}, and~\eqref{ineq:vjxij} that $\langle f^{\prime}(\bar{x}), d_{j_{k}} \rangle_{X^{\ast},X} + \frac{1}{2} \langle L_{xx}(\bar{v}) d_{j_{k}}, d_{j_{k}} \rangle_{X^{\ast}, X} \geq \frac{\eta \varepsilon}{8}$. This inequality and~\eqref{ineq:fLnu} derive $\langle f^{\prime}(x_{j_{k}}), d_{j_{k}} \rangle_{X^{\ast}, X} + \frac{1}{2} \langle L_{xx}(v_{j_{k}}) d_{j_{k}}, d_{j_{k}} \rangle_{X^{\ast},X} > \frac{\eta \varepsilon}{8} - \nu \Vert f^{\prime}(x_{j_{k}}) - f^{\prime}(\bar{x}) \Vert_{X^{\ast}} - \frac{\nu^{2}}{2} \Vert L_{xx}(v_{j_{k}}) - L_{xx}(\bar{v}) \Vert_{X \to X^{\ast}}$. Then, we have by~\eqref{ineq:vjxij} that
\begin{align*}
\limsup_{j \to \infty} \left[ \langle f^{\prime}(x_{j}), d_{j} \rangle_{X^{\ast}, X} + \frac{1}{2} \langle L_{xx}(v_{j}) d_{j}, d_{j} \rangle_{X^{\ast},X} \right] \geq \frac{\eta \varepsilon}{8} > 0,
\end{align*}
where note that any sequence $\{ a_{j} \} \subset \mathbb{R}$ and any subsequence $\{ a_{j_{k}} \} \subset \{ a_{j} \}$ satisfy $\limsup_{j \to \infty} a_{j} \geq \limsup_{k \to \infty} a_{j_{k}}$.
\end{itemize}
As a result, both cases~(a) and (b) yield contradictions to~\eqref{ineq:limsup_fL}. Therefore, we can verfy that $\Vert d(v) \Vert_{X} \to 0$ as $\Vert v - \bar{v} \Vert_{V} \to 0$.
\par
Next, we prove item~(ii). Theorem~\ref{theorem:eps_optimal} ensures that for any $v = (x, \lambda) \in V \backslash \Gamma$, there exist another $\sigma(v)^{6}$-optimal solution $\widehat{w}(v) = (\widehat{d}(v), \widehat{\mu}(v)) \in V$ of Q$(v)$ and $\rho(v) \in Y$ such that \eqref{ineq:sigma4_1},~\eqref{ineq:sigma4_3},~\eqref{ineq:sigma4_4}, and the following inequality hold:
\begin{align}
\dist( - f^{\prime}(x) - L_{xx}(v) \widehat{d}(v) - G^{\prime}(x)^{\ast} \rho(v), {\cal N}_{\nu B_{X}}(\widehat{d}(v)) ) \leq \sigma(v)^{3}. \label{ineq:sigma4_2_2}
\end{align}
Recall that $\sigma(v) \to 0$ and $\Vert d(v) \Vert_{X} \to 0$ as $v \to \bar{v}$. Then, we obtain by~\eqref{ineq:sigma4_1} that
\begin{align}
\Vert \widehat{d}(v) \Vert_{X} \leq \sigma(v)^{3} + \Vert d(v) \Vert_{X} \to 0  \quad (v \to \bar{v}). \label{lim:hat_d_0}
\end{align}
Hence, there exists $\gamma > 0$ such that $\widehat{d}(v) \in {\rm int}(\nu B_{X})$ for each $v \in B_{V}(\bar{v}, \gamma) \backslash \Gamma$. Since $\widehat{d}(v) \in {\rm int}(\nu B_{X})$ is equivalent to ${\cal N}_{\nu B_{X}}(\widehat{d}(v)) = \{ 0 \}$ from~\cite[Proposition~6.45]{BaCo11}, it is clear that $\dist(- f^{\prime}(x) - L_{xx}(v) \widehat{d}(v) - G^{\prime}(x)^{\ast} \rho(v), {\cal N}_{\nu B_{X}}(\widehat{d}(v)) ) = \Vert f^{\prime}(x) + L_{xx}(v) \widehat{d}(v) + G^{\prime}(x)^{\ast} \rho(v) \Vert_{X^{\ast}}$ for all $v \in B_{V}(\bar{v}, \gamma) \backslash \Gamma$. This equality and~\eqref{ineq:sigma4_2_2} lead to~\eqref{ineq:sigma4_2}.
\par
Finally, we show item~(i). Let $v \in B_{V}(\bar{v}, \gamma)\backslash \Gamma$ be arbitrary. We consider the following problem:
\begin{align*}
{\cal P}(u) \qquad
\begin{aligned}
& \displaystyle \mini_{d \in X} & & \langle f^{\prime}(x) - s, d \rangle_{X^{\ast}, X} + \frac{1}{2} \langle L_{xx}(v) d, d \rangle_{X^{\ast}, X}
\\
& \subj & & G(x) + G^{\prime}(x) d - t \in K,
\end{aligned}
\end{align*}
where $u \coloneqq (v, s, t) \in V \times X^{\ast} \times Y \eqqcolon U$. Let $\bar{u} \coloneqq (\bar{v}, 0, 0) \in U$. Notice that problem ${\cal P}(\bar{u})$ satisfies the SRCQ at $d = 0$. It then follows from~\cite[Lemma~4.44~and~Proposition~4.47]{BoSh00} that the sets of Lagrange multipliers regarding problem~${\cal P}(u)$, say ${\cal M}(d, u) \coloneqq \{ \rho \in Y; \, f^{\prime}(x) + L_{xx}(v) d + G^{\prime}(x)^{\ast} \rho - s = 0, ~ \rho \in {\cal N}_{K}(G(x) + G^{\prime}(x) d - t) \}$, satisfies ${\cal M}(0, \bar{u}) = \{ \bar{\lambda} \}$ and is upper Lipschitzian at $(0, \bar{u})$, i.e., there exist $m_{2} > 0$ and $r_{2} > 0$ such that for all $(d, u) \in B_{X \times U}((0, \bar{u}), r_{2})$,
\begin{align}
{\cal M}(d, u) \subset \{ \bar{\lambda} \} + m_{2} (\Vert d \Vert_{X} + \Vert u - \bar{u} \Vert_{U}) B_{Y}.  \label{subset:Lambda_UpperLip}
\end{align}
Let us define $u(v) \coloneqq (v, s(v), t(v))$, $s(v) \coloneqq f^{\prime}(x) + L_{xx}(v) \widehat{d}(v) + G^{\prime}(x)^{\ast} \rho(v)$, and $t(v) \coloneqq \sigma(v) ( \widehat{\mu}(v) - \lambda )$ for $v = (x, \lambda) \in B_{V}(\bar{v}, \gamma) \backslash \Gamma$. These definitions and~\eqref{ineq:sigma4_4} imply that for all $v \in B_{V}(\bar{v}, \gamma) \backslash \Gamma$,
\begin{align}
\rho(v) \in {\cal M}(\widehat{d}(v), u(v)). \label{in:eta_Lambda}
\end{align}
Combining item~(i) of~Lemma~\ref{lemma:three_results}, $v \in B_{V}(\bar{v}, \gamma) \backslash \Gamma$, and~\eqref{ineq:sigma4_1} derives $\Vert t(v) \Vert_{Y} \leq \sigma(v) \Vert \mu(v) - \widehat{\mu}(v) \Vert_{Y} + \sigma(v) \Vert \mu(v) \Vert_{Y} + \sigma(v) \Vert \bar{\lambda} \Vert_{Y} + \sigma(v) \Vert \lambda - \bar{\lambda} \Vert_{Y} \leq \sigma(v)^{4} + \sigma(v) \Vert \mu(v) \Vert_{Y} + \sigma(v) \Vert \bar{\lambda} \Vert_{Y} + \gamma \sigma(v) \to 0$ as $\Vert v - \bar{v} \Vert_{V} \to 0$. The fact, \eqref{ineq:sigma4_2}, and~\eqref{lim:hat_d_0} imply that if $\Vert v - \bar{v} \Vert_{V} \to 0$, then
\begin{align}
\Vert \widehat{d}(v) \Vert_{X} + \Vert u(v) - \bar{u} \Vert_{U} = \Vert \widehat{d}(v) \Vert_{X} + \Vert v - \bar{v} \Vert_{V} + \Vert s(v) \Vert_{X^{\ast}} + \Vert t(v) \Vert_{Y} \to 0. \label{lim:xivuv}
\end{align}
Thus, there exists $r_{3} \in (0, \gamma)$ such that $(\widehat{d}(v), u(v)) \in B_{X \times U}((0, \bar{u}), r_{2})$ for any $v \in B_{V}(\bar{v}, r_{3}) \backslash \Gamma$. It then follows from~\eqref{subset:Lambda_UpperLip} and~\eqref{in:eta_Lambda} that for $v \in B_{V}(\bar{v}, r_{3}) \backslash \Gamma$,
\begin{align}
\rho(v) \in \{ \bar{\lambda} \} + m_{2} (\Vert \widehat{d}(v) \Vert_{X} + \Vert u(v) - \bar{u} \Vert_{U}) B_{Y}. \label{subset:eta_bounded}
\end{align}
Since the triangle inequality ensures $\Vert \mu(v) - \bar{\lambda} \Vert_{Y} \leq \Vert \mu(v) - \widehat{\mu}(v) \Vert_{Y} + \Vert \widehat{\mu}(v) - \rho(v) \Vert_{Y} + \Vert \rho(v) - \bar{\lambda} \Vert_{Y}$, combining~\eqref{ineq:sigma4_1},~\eqref{ineq:sigma4_3}, and~\eqref{subset:eta_bounded} derives
\begin{align*}
\Vert \mu(v) - \bar{\lambda} \Vert_{Y} \leq \sigma(v)^{3} + \sigma(v)^{2} + m_{2} ( \Vert \widehat{d}(v) \Vert_{X} + \Vert u(v) - \bar{u} \Vert_{U} )
\end{align*}
for $v \in B_{V}(\bar{v}, r_{3}) \backslash \Gamma$. Hence, using~\eqref{lim:xivuv} yields $\Vert \mu(v) - \bar{\lambda} \Vert_{Y} \to 0$ as $\Vert v - \bar{v} \Vert_{V} \to 0$. This result and the former part of the proof guarantee that $\Vert w(v) - \bar{w} \Vert_{V} = \Vert d(v) \Vert_{X} + \Vert \mu(v) - \bar{\lambda} \Vert_{Y} \to 0$ as $\Vert v - \bar{v} \Vert_{V} \to 0$.
\end{proof}

\noindent
By exploiting the above results, the following theorem guarantees the solvability of problem~P$(v)$ and provides an inequality that is crucial for analyzing the local convergence of Algorithm~\ref{Local_SSQP}.
\begin{theorem} \label{theorem:omega_sigma}
Suppose that {\rm (A1)} and {\rm (A2)} hold. Let $w(v) = (d(v), \mu(v))$ denote a $\sigma(v)^{6}$-optimal solution of ${\rm Q}(v)$ for any $v = (x, \lambda) \in V \backslash \Gamma$. Then, there exist $m > 0$ and $r > 0$ such that for all $v \in B_{V}(\bar{v}, r) \backslash \Gamma$, the $\sigma(v)^{6}$-optimal solution $w(v)$ satisfies $\Vert w(v) - \bar{w} \Vert_{V} \leq m \sigma(v)$ and is an approximate local optimum of ${\rm P}(v)$.
\end{theorem}

\begin{proof}
We notice that $\bar{w} = (0, \bar{\lambda}) \in V$ is a KKT point of P$_{0}(\bar{v})$ and satisfies the SRCQ and the SOSC for P$_{0}(\bar{v})$. Thus, we have from~\cite[Theorems~3.1 and 3.2]{KaSt18} that
there exist $m_{1} > 0$ and $r_{1} > 0$ such that for $(\alpha, \beta) \in r_{1} B_{X^{\ast} \times Y}$ and $(d, \rho) \in {\cal S}(\alpha, \beta)$,
\begin{align}
\Vert d \Vert_{X} + \Vert \rho - \bar{\lambda} \Vert_{Y} \leq m_{1} (\Vert \alpha \Vert_{X^{\ast}} + \Vert \beta \Vert_{Y}),  \label{ineq:xi_eta_albe}
\end{align}
where the set ${\cal S}(\alpha, \beta)$ is defined as
\begin{align*}
{\cal S}(\alpha, \beta) \coloneqq \left\{ (d, \rho) \in r_{1} B_{X} \times Y; \,
\begin{gathered}
f^{\prime}(\bar{x}) + L_{xx}(\bar{v}) d + G(\bar{x})^{\ast} \rho - \alpha = 0, 
\\
\rho \in {\cal N}_{K}(G(\bar{x}) + G^{\prime}(\bar{x}) d - \beta)
\end{gathered}
\, \right\}.
\end{align*}
\par
Recall that $f^{\prime}$, $G$, and $G^{\prime}$ are locally Lipschitz continuous at $\bar{v}$, and that (A3) implies $L_{xx}$ is locally Lipschitz continuous at $\bar{v}$. From these facts and Proposition~\ref{prop:error_bound}, there exist $m_{2} > 0$ and $r_{2} > 0$ such that for all $v = (x,\lambda) \in B_{V}(\bar{v}, r_{2})$,
\begin{align}
\begin{aligned} \label{ineq:fGLxx}
& \Vert f^{\prime}(x) - f^{\prime}(\bar{x}) \Vert_{X^{\ast}} \leq m_{2} \sigma(v), && \Vert G(x) - G(\bar{x}) \Vert_{Y} \leq m_{2} \sigma(v),
\\
& \Vert G^{\prime}(x) - G^{\prime}(\bar{x}) \Vert_{X \to Y} \leq m_{2} \sigma(v), && \Vert L_{xx}(v) - L_{xx}(\bar{v}) \Vert_{X \to X^{\ast}} \leq m_{2} \sigma(v). 
\end{aligned}
\end{align}
\par
Item~(ii) of Proposition~\ref{proposition:xi_zero} ensures that there exists $r_{3} > 0$ such that for any $v \in B_{V}(\bar{v}, r_{3}) \backslash \Gamma$, there exist another $\sigma(v)^{6}$-optimal solution $\widehat{w}(v) = (\widehat{d}(v), \widehat{\mu}(v)) \in V$ and $\rho(v) \in Y$ such that~\eqref{ineq:sigma4_1}--\eqref{ineq:sigma4_4} are satisfied. We define 
\begin{align}
m_{3} \coloneqq 1 + \Vert \bar{\lambda} \Vert_{Y}, \quad m_{4} \coloneqq r_{1} + 3 m_{2} + 2 m_{3} + 2 r_{1} m_{2} + m_{2} m_{3}. \label{def:m3m4}
\end{align}
Note that $\sigma(v) \to 0$ as $\Vert v - \bar{v} \Vert_{V} \to 0$. By this fact and items~(i) and~(ii) of Proposition~\ref{proposition:xi_zero}, we have $\Vert \widehat{d}(v) \Vert_{X} \to 0$ and $\Vert \widehat{\mu}(v) - \bar{\lambda} \Vert_{Y} \to 0$ as $\Vert v - \bar{v} \Vert_{V} \to 0$. According to these facts, there exists $r_{4} > 0$ such that for all $v \in B_{V}(\bar{v}, r_{4}) \backslash \Gamma$,
\begin{align}
\sigma(v) \leq \min \left\{ 1, \, \frac{r_{1}}{2 m_{4}}, \frac{\nu}{4} \right\}, \quad \Vert \widehat{d}(v) \Vert_{X} \leq \min \left\{ r_{1}, \frac{\nu}{4} \right\}, \quad \Vert \widehat{\mu}(v) \Vert_{Y} \leq m_{3}. \label{ineq:sigma_xizeta}
\end{align}
\par
In the following, we show the assertion of this theorem. Let $m > 0$ and $r > 0$ be defined by
\begin{align}
m \coloneqq 2 + m_{1} m_{4}, \quad r \coloneqq \min \{ r_{1}, r_{2}, r_{3}, r_{4} \}, \label{def:constants_mr}
\end{align}
respectively. We take $v \in B_{V}(\bar{v}, r) \backslash \Gamma$ arbitrarily and define $\alpha(v)$ and $\beta(v)$ as follows:
\begin{align}
\begin{aligned} \label{def:alpha_beta}
\alpha(v) &\coloneqq f^{\prime}(\bar{x}) + L_{xx}(\bar{v}) \widehat{d}(v) + G^{\prime}(\bar{x})^{\ast} \rho(v), 
\\
\beta(v) &\coloneqq G(\bar{x}) - G(x) + G^{\prime}(\bar{x}) \widehat{d}(v) - G^{\prime}(x) \widehat{d}(v) + \sigma(v) (\widehat{\mu}(v) - \lambda). 
\end{aligned}
\end{align}
Exploiting~\eqref{ineq:sigma4_2},~\eqref{ineq:sigma4_3},~\eqref{ineq:fGLxx}, and~\eqref{ineq:sigma_xizeta} derives
\begin{align*}
\Vert \alpha(v) \Vert_{X^{\ast}} 
& \leq \Vert f^{\prime}(x) + L_{xx}(v) \widehat{d}(v) + G^{\prime}(x)^{\ast} \rho(v) \Vert_{X^{\ast}} + \Vert f^{\prime}(x) - f^{\prime}(\bar{x}) \Vert_{X^{\ast}}
\\
&\qquad + \Vert L_{xx}(v) - L_{xx}(\bar{v}) \Vert_{X \to X^{\ast}} \Vert \widehat{d}(v) \Vert_{X} + \Vert G^{\prime}(x) - G^{\prime}(\bar{x}) \Vert_{X \to Y} \Vert \rho(v) \Vert_{Y}
\\
& \leq \sigma(v)^{3} + m_{2} \sigma(v) + r_{1} m_{2} \sigma(v) + m_{2} \sigma(v) (\Vert \rho(v) - \widehat{\mu}(v) \Vert_{Y} + \Vert \widehat{\mu}(v) \Vert_{Y})
\\
& \leq (1 + 2 m_{2} + r_{1} m_{2} + m_{2} m_{3}) \sigma(v),
\end{align*}
and
\begin{align*}
\Vert \beta(v) \Vert_{Y} 
& \leq \Vert G(x) - G(\bar{x}) \Vert_{Y} + \Vert G^{\prime}(x) - G^{\prime}(\bar{x}) \Vert_{X \to Y} \Vert \widehat{d}(v) \Vert_{X} 
\\
& \hspace{34.85mm} + \sigma(v) (\Vert \widehat{\mu}(v) \Vert_{Y} + \Vert \lambda - \bar{\lambda} \Vert_{Y} + \Vert \bar{\lambda} \Vert_{Y})
\\
& \leq ( m_{2} + r_{1} m_{2} + m_{3} + r_{1} + \Vert \bar{\lambda} \Vert_{Y} ) \sigma(v),
\end{align*}
where note that $v \in B_{V}(\bar{v}, r)$ and the second equality of~\eqref{def:constants_mr} yield $\Vert \lambda - \bar{\lambda} \Vert_{Y} \leq r_{1}$, and it is used in the last inequality.
It then follows from~\eqref{def:m3m4} and \eqref{ineq:sigma_xizeta} that
\begin{align}
\Vert \alpha(v) \Vert_{X^{\ast}} + \Vert \beta(v) \Vert_{Y} \leq m_{4} \sigma(v) \leq \frac{r_{1}}{2} < r_{1}, ~~ \mbox{i.e.,} ~~ (\alpha(v), \beta(v)) \in r_{1} B_{X^{\ast} \times Y}. \label{ineq:alpha_beta}
\end{align}
Now, combining~\eqref{ineq:sigma4_4},~\eqref{ineq:sigma_xizeta}, and~\eqref{def:alpha_beta} implies $(\widehat{d}(v), \rho(v)) \in {\cal S}(\alpha(v), \beta(v))$. This fact,~\eqref{ineq:xi_eta_albe}, and~\eqref{ineq:alpha_beta} lead to
\begin{align}
\Vert \widehat{d}(v) \Vert_{X} + \Vert \rho(v) - \bar{\lambda} \Vert_{Y} \leq m_{1} (\Vert \alpha(v) \Vert_{X^{\ast}} + \Vert \beta(v) \Vert_{Y}). \label{ineq:xi_hatd_rho}
\end{align}
By~\eqref{ineq:alpha_beta} and~\eqref{ineq:xi_hatd_rho}, we obtain
\begin{align}
\Vert w(v) - \bar{w} \Vert_{V} 
&\leq \Vert \widehat{d}(v) \Vert_{X} + \Vert \rho(v) - \bar{\lambda} \Vert_{Y} + \Vert d(v) - \widehat{d}(v) \Vert_{X} + \Vert \mu(v) - \rho(v) \Vert_{Y} \nonumber
\\
&\leq m_{1} m_{4} \sigma(v) + \Vert w(v) - \widehat{w}(v) \Vert_{V} + \Vert \rho(v) - \widehat{\mu}(v) \Vert_{Y}. \label{ineq:omega_bar}
\end{align}
Meanwhile, the first inequality of~\eqref{ineq:sigma_xizeta} yields $\sigma(v)^{3} \leq \sigma(v)^{2} \leq \sigma(v)$. Then, using~\eqref{ineq:sigma4_1} and~\eqref{ineq:sigma4_3} implies $\Vert w(v) - \widehat{w}(v) \Vert_{V} + \Vert \rho(v) - \widehat{\mu}(v) \Vert_{Y} \leq 2 \sigma(v)$.
These facts, the first equality of~\eqref{def:constants_mr}, and~\eqref{ineq:omega_bar} derive $\Vert w(v) - \bar{w} \Vert_{V} \leq m \sigma(v)$.
\par
From~\eqref{ineq:sigma4_1} and~\eqref{ineq:sigma_xizeta}, it is clear that $\Vert d(v) \Vert_{X} \leq \Vert \widehat{d}(v) \Vert_{X} + \sigma(v) \leq \frac{1}{2} \nu < \nu$. This fact means that the constraint $\Vert d \Vert_{X} \leq \nu$ in Q$(v)$ is superfluous, and hence $w(v)$ is an approximate local optimum of P$(v)$.
\end{proof}

\par
Hereafter, we focus on the local and quadratic convergence of Algorithm~\ref{Local_SSQP}. First, we prepare a corollary of Theorem~\ref{theorem:omega_sigma}, which provides some inequalities possessing suitable forms for the subsequent analysis compared to that of Theorem~\ref{theorem:omega_sigma}. Second, by utilizing these inequalities and the Lipschitz continuity of the second derivatives of $f$ and $G$, we prepare a lemma that plays a crucial role in the analysis of the quadratic convergence. Finally, we prove the local and quadratic convergence.

\begin{corollary} \label{coro:solvable}
Suppose that {\rm (A1)} and {\rm (A2)} hold. For any $v = (x, \lambda) \in V \backslash \Gamma$, let $w(v) = (d(v), \mu(v))$ denote a $\sigma(v)^{6}$-optimal solution of ${\rm Q}(v)$, and let $p(v) = ( d(v), \mu(v) - \lambda )$. Then, there exist $m > 0$ and $r > 0$ such that for all $v \in B_{V}(\bar{v}, r) \backslash \Gamma$, the $\sigma(v)^{6}$-optimal solution $w(v)$ is an approximate local optimum of~${\rm P}(v)$ and satisfies
\begin{align*}
\begin{aligned}
\Vert p(v) \Vert_{V} \leq m \sigma(v), \quad \Vert d(v) \Vert_{X} \leq m \sigma(v), \quad \Vert \mu(v) - \lambda \Vert_{Y} \leq m \sigma(v).
\end{aligned}
\end{align*}
\end{corollary}

\begin{proof}
Proposition~\ref{prop:error_bound} and Theorem~\ref{theorem:omega_sigma} ensure that there exist $b > 0$ and $r > 0$ such that for any $v \in B_{V}(\bar{v}, r) \backslash \Gamma$,
\begin{align}
\Vert v - \bar{v} \Vert_{V} \leq b \sigma(v), \quad \Vert w(v) - \bar{w} \Vert_{V} \leq b \sigma(v), \label{ineq:omega_sigma}
\end{align}
and $w(v)$ is an approximate local optimum of~P$(v)$ whenever $v \in B_{V}(\bar{v}, r) \backslash \Gamma$. Let $m \coloneqq 2 b$ and let $v \in B_{V}(\bar{v}, r) \backslash \Gamma$. Combining~\eqref{ineq:omega_sigma} and the definition of $p(v)$ yields $\Vert p(v) \Vert_{V} \leq \Vert w(v) - \bar{w} \Vert_{V} + \Vert \lambda - \bar{\lambda} \Vert_{Y} \leq m \sigma(v)$.
Since $\Vert p(v) \Vert_{V} = \Vert d(v) \Vert_{X} + \Vert \mu(v) - \lambda \Vert_{Y}$ holds, we have $\Vert d(v) \Vert_{X} \leq m \sigma(v)$ and $\Vert \mu(v) - \lambda \Vert_{Y} \leq m \sigma(v)$. 
\end{proof}


\begin{lemma} \label{lem:sigma_property}
Suppose that {\rm (A1)}, {\rm (A2)}, and {\rm (A3)} are satisfied. For each $v = (x, \lambda) \in V \backslash \Gamma$, let $w(v) = (d(v), \mu(v))$ denote a $\sigma(v)^{6}$-optimal solution of~${\rm Q}(v)$, and let $p(v) = ( d(v), \mu(v) - \lambda )$. Then, there exist $m > 0$ and $r > 0$ such that $\sigma(v + p(v)) \leq m \sigma(v)^{2}$ for all $v \in B_{V}(\bar{v}, r) \backslash \Gamma$.
\end{lemma}

\begin{proof}
It follows from Corollary~\ref{coro:solvable} that there exist $m_{1} > 0$ and $r_{1} > 0$ such that for each $v = (x, \lambda) \in B_{V}(\bar{v}, r_{1}) \backslash \Gamma$, 
\begin{align} \label{ineq:delta_v}
\begin{aligned}
\Vert p(v) \Vert_{V} \leq m_{1} \sigma(v), \quad \Vert d(v) \Vert_{X} \leq m_{1} \sigma(v), \quad \Vert \mu(v) - \lambda \Vert_{Y} \leq m_{1} \sigma(v).
\end{aligned}
\end{align} 
Now, we recall that $G^{\prime}$ is locally Lipschitz continuous at $\bar{x}$. Moreover, assumption (A3) implies that $L_{xx}$ is locally Lipschitz continuous at $\bar{v}$. From these facts, there exist $m_{2} > 0$ and $r_{2} > 0$ such that for $v = (x, \lambda) \in B_{V}(\bar{v}, r_{2})$ and $\widetilde{v} = (\widetilde{x}, \widetilde{\lambda}) \in B_{V}(\bar{v}, r_{2})$,
\begin{align}
& \Vert L_{x}(v) - L_{x}(\widetilde{v}) - L_{xx}(\widetilde{v}) (x - \widetilde{x}) - G^{\prime}(\widetilde{x})^{\ast} (\lambda - \widetilde{\lambda}) \Vert_{X^{\ast}} \leq m_{2} \Vert v - \widetilde{v} \Vert_{V}^{2},  \label{ineq:nab_Lag}
\\
& \Vert L_{xx}(v) - L_{xx}(\widetilde{v}) \Vert_{X \to X^{\ast}} \leq m_{2} \Vert v - \widetilde{v} \Vert_{V}, \label{ineq:nab_Lag2}
\\
& \Vert G(x) - G(\widetilde{x}) - G^{\prime}(\widetilde{x}) (x - \widetilde{x}) \Vert_{Y} \leq m_{2} \Vert x - \widetilde{x} \Vert_{X}^{2}, \label{ineq:func_G}
\\
& \Vert G^{\prime}(x) - G^{\prime}(\widetilde{x}) \Vert_{X \to Y} \leq m_{2} \Vert x - \widetilde{x} \Vert_{X}. \label{ineq:func_G2}
\end{align}
Let $m_{3} \coloneqq \max \{ \Vert L_{xx}(\bar{v}) \Vert_{X \to X^{\ast}}, \Vert G^{\prime}(\bar{x}) \Vert_{X \to Y} \} + m_{2} r_{2}$. According to~\eqref{ineq:nab_Lag2} and~\eqref{ineq:func_G2}, we obtain
\begin{align}
\begin{gathered}
\Vert L_{xx}(v) \Vert_{X \to X^{\ast}} \leq \Vert L_{xx}(\bar{v}) \Vert_{X \to X^{\ast}} + m_{2} r_{2} \leq m_{3}, 
\\
\Vert G^{\prime}(x) \Vert_{X \to Y} \leq \Vert G^{\prime}(\bar{x}) \Vert_{X \to Y} + m_{2} r_{2} \leq m_{3} \label{ineq:bounded_LxxG}
\end{gathered}
\end{align}
for $v = (x, \lambda) \in B_{V}(\bar{v}, r_{2})$. From item~(ii) of Proposition~\ref{proposition:xi_zero}, there exists $r_{3} > 0$ such that for any $v \in B_{V}(\bar{v}, r_{3}) \backslash \Gamma$, problem~Q$(v)$ has another $\sigma(v)^{6}$-optimal solution $\widehat{w}(v) = ( \widehat{d}(v), \widehat{\mu}(v) ) \in V$, and there exists $\rho(v) \in Y$ satisfying~\eqref{ineq:sigma4_1}--\eqref{ineq:sigma4_4}. 
Note that $\sigma(v) \to 0$ as $v \to \bar{v}$. Then, there exists $r_{4} > 0$ such that
\begin{align}
\sigma(v) \leq \min \left \{ 1, \frac{r_{1}}{2m_{1}}, \frac{r_{2}}{2m_{1}}, \frac{r_{3}}{2m_{1}} \right\} \label{ineq:delta_v2}
\end{align}
for any $v \in B_{V}(\bar{v}, r_{4}) \backslash \Gamma$.
\par
By exploiting the above results, the assertion of this Lemma is proven in the following. We define $m > 0$ and $r > 0$ as
\begin{align}
m \coloneqq 5 + 2 m_{1} + 4 m_{3} + 3 m_{1}^{2} m_{2}, \quad r \coloneqq \min \left\{ \frac{r_{1}}{2}, \frac{r_{2}}{2}, \frac{r_{3}}{2}, r_{4} \right\}, \label{def:fixed_mr}
\end{align}
respectively. Let $v = (x, \lambda) \in B_{V}(\bar{v}, r) \backslash \Gamma$ be arbitrary. The first ineqaulity of~\eqref{ineq:delta_v}, inequality~\eqref{ineq:delta_v2}, and the second equality of~\eqref{def:fixed_mr} lead to
\begin{align}
\Vert v + p(v) - \bar{v} \Vert_{V} \leq \frac{1}{2}\min \{ r_{1}, r_{2}, r_{3} \} + m_{1} \sigma(v) \leq \min \{ r_{1}, r_{2}, r_{3} \}. \label{ineq:v_vplus}
\end{align}
Using $L_{x}(v) = f^{\prime}(x) + G^{\prime}(x)^{\ast} \lambda$ derives
\begin{align*}
& L_{x}(v + p(v)) 
\\
& = L_{x}(v + p(v)) - L_{x}(v) - L_{xx}(v) d(v) - G^{\prime}(x)^{\ast} (\mu(v) - \lambda) + L_{xx}(v) \widehat{d}(v) + f^{\prime}(x) + G^{\prime}(x)^{\ast} \rho(v)
\\
& \hspace{24mm} + L_{x}(v) + L_{xx}(v) d(v) + G^{\prime}(x)^{\ast} (\mu(v) - \lambda) - L_{xx}(v) \widehat{d}(v) - f^{\prime}(x) - G^{\prime}(x)^{\ast} \rho(v)
\\
& = L_{x}(v + p(v)) - L_{x}(v) - L_{xx}(v) d(v) - G^{\prime}(x)^{\ast} (\mu(v) - \lambda) + L_{xx}(v) \widehat{d}(v) + f^{\prime}(x) + G^{\prime}(x)^{\ast} \rho(v)
\\
& \hspace{73.225mm} + L_{xx}(v) (d(v) - \widehat{d}(v)) + G^{\prime}(x)^{\ast} (\mu(v) - \rho(v)).
\end{align*}
It then follows from~\eqref{ineq:sigma4_2},~\eqref{ineq:nab_Lag},~\eqref{ineq:bounded_LxxG}, and~\eqref{ineq:v_vplus} that
\begin{align}
\Vert L_{x}(v + p(v)) \Vert_{X^{\ast}} 
&\leq m_{2} \Vert p(v) \Vert_{X}^{2} + m_{3} \Vert d(v) - \widehat{d}(v) \Vert_{X} + m_{3} \Vert \mu(v) - \rho(v) \Vert_{Y} + \sigma(v)^{3} \nonumber
\\
&\leq m_{2} \Vert p(v) \Vert_{X}^{2} + m_{3} \Vert w(v) - \widehat{w}(v) \Vert_{X} + m_{3} \Vert \widehat{\mu}(v) - \rho(v) \Vert_{Y} + \sigma(v)^{3}. \label{ineq:Lx_vpv}
\end{align}
By \eqref{ineq:sigma4_1},~\eqref{ineq:sigma4_3},~\eqref{ineq:delta_v},~\eqref{ineq:delta_v2}, and~\eqref{ineq:v_vplus}, we can evaluate~\eqref{ineq:Lx_vpv} as follows:
\begin{align}
\Vert L_{x}(v + p(v)) \Vert_{X^{\ast}} 
& \leq m_{1}^{2} m_{2} \sigma(v)^{2} + \sigma(v)^{3} + m_{3} \sigma(v)^{3} + m_{3} \sigma(v)^{2} \nonumber
\\
& \leq (1 + 2 m_{3} + m_{1}^{2} m_{2}) \sigma(v)^{2}. \label{ineq:sigma_former}
\end{align}
For simplicity, we denote
\begin{align*}
R(v) &\coloneqq G(x) + G^{\prime}(x) \widehat{d}(v) - \sigma(v) (\widehat{\mu}(v) - \lambda) + \rho(v),
\\
S(v) &\coloneqq G(x + d(v)) + \mu(v),
\\
T(v) &\coloneqq G(x + d(v)) - G(x) - G^{\prime}(x) d(v).
\end{align*}
The definition of $R(v)$ and~\eqref{ineq:sigma4_4} yield $G(x) + G^{\prime}(x)\widehat{d}(v) - \sigma(v)(\widehat{\mu}(v) - \lambda) = P_{K}(R(v))$. From the equality and the definitions of $S(v)$ and $T(v)$, we have
\begin{align*}
& G(x + d(v)) - P_{K}(G(x + d(v)) + \mu(v))
\\
& = G(x) + G^{\prime}(x) d(v) + T(v) - P_{K}(S(v))
\\
& = G(x) + G^{\prime}(x) \widehat{d}(v) - \sigma(v) ( \widehat{\mu}(v) - \lambda )  - G^{\prime}(x) \widehat{d}(v) + \sigma(v) ( \widehat{\mu}(v) - \lambda ) + G^{\prime}(x) d(v) + T(v) - P_{K}(S(v))
\\
& = P_{K}(R(v)) - P_{K}(S(v)) + T(v) + G^{\prime}(x)( d(v) - \widehat{d}(v) )  + \sigma(v)(\widehat{\mu}(v) - \mu(v)) + \sigma(v)(\mu(v) - \lambda).
\end{align*}
It then follows from the nonexpansiveness of $P_{K}$ that
\begin{align}
& \Vert G(x + d(v)) - P_{K}(G(x + d(v)) + \mu(v)) \Vert_{Y} \nonumber
\\
& \leq \Vert R(v) - S(v) \Vert_{Y} + \Vert T(v) \Vert_{Y} + \Vert G^{\prime}(x) \Vert_{X \to Y} \Vert d(v) - \widehat{d}(v) \Vert_{X} + \sigma(v) \Vert \mu(v) - \widehat{\mu}(v) \Vert_{Y} + \sigma(v) \Vert \mu(v) - \lambda \Vert_{Y}. \label{ineq:GPK}
\end{align}
In what follows, we evaluate $\Vert R(v) - S(v) \Vert_{Y}$. Note that
\begin{align*}
R(v) - S(v) 
& = G(x) + G^{\prime}(x) d(v) - G(x + d(v))  - G^{\prime}(x)d(v) + G^{\prime}(x) \widehat{d}(v) - \sigma(v) (\widehat{\mu}(v) - \lambda) + \rho(v) - \mu(v)
\\
& = - T(v) + G^{\prime}(x) ( \widehat{d}(v) - d(v) ) - \sigma(v) (\widehat{\mu}(v) - \mu(v) ) - \sigma(v) ( \mu(v) - \lambda)  + \rho(v) - \widehat{\mu}(v) + \widehat{\mu}(v) - \mu(v),
\end{align*}
and hence
\begin{align}
\Vert R(v) - S(v) \Vert_{Y} 
& \leq \Vert T(v) \Vert_{Y} + \Vert G^{\prime}(x) \Vert_{X \to Y} \Vert d(v) - \widehat{d}(v) \Vert_{X} \nonumber
\\
& \qquad + \sigma(v) \Vert \mu(v) - \widehat{\mu}(v) \Vert_{Y} + \sigma(v) \Vert \mu(v) - \lambda \Vert_{Y} + \Vert \rho(v) - \widehat{\mu}(v) \Vert_{Y} + \Vert \widehat{\mu}(v) - \mu(v) \Vert_{Y}.  \label{ineq:RvSv}
\end{align}
Now, combining~\eqref{ineq:func_G},~\eqref{ineq:v_vplus}, and the definition of $T(v)$ yields
\begin{align}
\Vert T(v) \Vert_{Y} \leq m_{2} \Vert d(v) \Vert_{X}^{2}. \label{ineq:Tv}
\end{align}
Substituting~\eqref{ineq:RvSv} and~\eqref{ineq:Tv} into~\eqref{ineq:GPK} derives
\begin{align*}
& \Vert G(x + d(v)) - P_{K}(G(x + d(v)) + \mu(v)) \Vert_{Y}
\\
& \leq 2 m_{2} \Vert d(v) \Vert_{X}^{2} + 2 \Vert G^{\prime}(x) \Vert_{X \to Y} \Vert d(v) - \widehat{d}(v) \Vert_{X} 
\\
& \, \quad + 2\sigma(v) \Vert \mu(v) - \widehat{\mu}(v) \Vert_{Y} + 2\sigma(v) \Vert \mu(v) - \lambda \Vert_{Y} + \Vert \rho(v) - \widehat{\mu}(v) \Vert_{Y} + \Vert \widehat{\mu}(v) - \mu(v) \Vert_{Y}.
\end{align*}
It then follows from~\eqref{ineq:sigma4_1},~\eqref{ineq:sigma4_3},~\eqref{ineq:delta_v},~\eqref{ineq:bounded_LxxG},~\eqref{ineq:delta_v2}, and~\eqref{ineq:v_vplus} that
\begin{align}
& \Vert G(x + d(v)) - P_{K}(G(x + d(v)) + \mu(v)) \Vert_{Y} \nonumber
\\
&\leq 2 m_{1}^{2} m_{2} \sigma(v)^{2} + 2 m_{3} \sigma(v)^{3} + 2 \sigma(v)^{4} + 2 m_{1} \sigma(v)^{2} + \sigma(v)^{2} + \sigma(v)^{3} \nonumber
\\
&\leq (4 + 2 m_{1} + 2 m_{3} + 2 m_{1}^{2} m_{2}) \sigma(v)^{2}. \label{ineq:sigma_latter}
\end{align}
Exploiting~\eqref{def:fixed_mr}, \eqref{ineq:sigma_former}, and~\eqref{ineq:sigma_latter} implies $\sigma(v + p(v)) \leq m \sigma(v)^{2}$. 
\end{proof}

\noindent
Finally, we prove the local and quadratic convergence of Algorithm~\ref{Local_SSQP}.
\begin{theorem}
Suppose that {\rm (A1)}, {\rm (A2)}, and {\rm (A3)} are satisfied. Then, there exists $r > 0$ such that if the initial point $v_{0} \in V$ of Algorithm~{\rm \ref{Local_SSQP}} satisfies $v_{0} \in B_{V}(\bar{v}, r) \backslash \Gamma$, then Algorithm~{\rm \ref{Local_SSQP}} is well defined, and any generated sequence $\{ v_{k} \}$, which satisfies $v_{k} \not \in \Gamma$ for $k \in \mathbb{N} \cup \{ 0 \}$, converges quadratically to $\bar{v}$.
\end{theorem}

\begin{proof}
For each $v = (x, \lambda) \in V \backslash \Gamma$, let $w(v) = (d(v), \mu(v))$ denote a $\sigma(v)^{6}$-optimal solution of problem~Q$(v)$, and let $p(v) = (d(v), \mu(v) - \lambda)$. From Propositions~\ref{prop:error_bound} and~\ref{prop:sigma_Lipschitz}, Corollary~\ref{coro:solvable}, and Lemma~\ref{lem:sigma_property}, there exist $b > 1$ and $\gamma \in (0,1]$ such that for all $v \in B_{V}(\bar{v}, \gamma) \backslash \Gamma$,
\begin{align}
\begin{aligned} \label{ineq:4kinds}
& 
\begin{aligned}
& \Vert v - \bar{v} \Vert_{V} \leq b \sigma(v), && \sigma(v) \leq b \Vert v - \bar{v} \Vert_{V},
\\
& \Vert p(v) \Vert_{V} \leq b \sigma(v), && \sigma(v + p(v)) \leq b \sigma(v)^{2},
\end{aligned}
\\
&\mbox{$w(v)$ is an approximate local optimum of P$(v)$.}
\end{aligned}
\end{align}
Since $\sigma(v) \to 0$ as $v \to \bar{v}$, there exists $r > 0$ such that for any $v \in B_{V}(\bar{v},r)$,
\begin{align}
r \leq \frac{b^{2} - 1}{b^{2}} \gamma, \qquad \sigma(v) \leq \frac{\gamma}{2b^{4}}. \label{ineq:r_sigma}
\end{align}
\par
From now on, we discuss the well-definedness of Algorithm~\ref{Local_SSQP}. Let $\{ v_{k} \}$ be a sequence generated by $v_{k+1} = v_{k} + p(v_{k})$ for $k \in \mathbb{N} \cup \{ 0 \}$ and $v_{0} \in B_{V}(\bar{v}, r) \backslash \Gamma$. Assume that $v_{k} \not \in \Gamma$ for each $k \in \mathbb{N} \cup \{ 0 \}$. Let us define $\sigma_{k} \coloneqq \sigma(v_{k})$, $p_{k} \coloneqq p(v_{k})$, and $w_{k} \coloneqq w(v_{k})$ for $k \in \mathbb{N} \cup \{ 0 \}$. We prove the following statement by mathematical induction: For all $k \in \mathbb{N} \cup \{ 0 \}$,
\begin{align}
v_{k} \in B_{V}(\bar{v}, \gamma) \backslash \Gamma, \qquad \sigma_{k} \leq \frac{\gamma}{2 b^{3}}. \label{ineq:math_induction}
\end{align}
Recall that $v_{0} \in B_{V}(\bar{v}, r) \backslash \Gamma$. Since $r \leq \gamma$ holds from $b > 1$ and the first inequality of~\eqref{ineq:r_sigma}, we have $v_{0} \in B_{V}(\bar{v}, \gamma) \backslash \Gamma$. Using $b > 1$, $v_{0} \in B_{V}(\bar{v}, r)$, and the second inequality of~\eqref{ineq:r_sigma} implies $\sigma_{0} \leq \frac{\gamma}{2 b^{4}} \leq \frac{\gamma}{2 b^{3}}$, and thus~\eqref{ineq:math_induction} holds for $k = 0$. Next, we assume that $v_{\ell} \in B_{V}(\bar{v}, \gamma) \backslash \Gamma$ and $\sigma_{\ell} \leq \frac{\gamma}{2 b^{3}}$ for all $\ell \in \{ 0, 1, \ldots, k \}$. It then follows from $\frac{\gamma}{2 b^{2}} \leq \frac{1}{2}$ and the fourth inequality of~\eqref{ineq:4kinds} that for any $\ell \in \{ 1, 2, \ldots, k \}$,
\begin{align}
\sigma_{\ell} \leq b \sigma_{\ell-1}^{2} \leq \frac{1}{2} \sigma_{\ell-1} \leq \frac{1}{2} b \sigma_{\ell-2}^{2} \leq \frac{1}{2^{2}} \sigma_{\ell-2} \leq \cdots \leq \frac{1}{2^{\ell}} \sigma_{0} \leq \frac{\gamma}{2^{\ell+1} b^{3}}. \label{ineq:sigma_ell}
\end{align}
Hence, combining the last inequality of~\eqref{ineq:sigma_ell}, the third inequality of~\eqref{ineq:4kinds}, and the first inequality of~\eqref{ineq:r_sigma} yields
\begin{align}
\Vert v_{k+1} - \bar{v} \Vert_{V}
& \leq \Vert v_{k} - \bar{v} \Vert_{V} + \Vert p_{k} \Vert_{V} \leq \Vert v_{k-1} - \bar{v} \Vert_{V} + \Vert p_{k-1} \Vert_{V} + \Vert p_{k} \Vert_{V} \nonumber
\\
& \leq \cdots \leq \Vert v_{0} - \bar{v} \Vert_{V} + \sum_{\ell=0}^{k} \Vert p_{\ell} \Vert_{V} \leq r + b \sum_{\ell=0}^{k} \sigma_{\ell}  \leq r + \frac{\gamma}{b^{2}} \left( 1 - \frac{1}{2^{k+1}} \right) \leq r + \frac{\gamma}{b^{2}} \leq \gamma. \label{ineq:vkplus1_vbar}
\end{align}
Since $v_{k+1} \not \in \Gamma$ holds by the assumption, the last inequality of~\eqref{ineq:vkplus1_vbar} guarantees $v_{k+1} \in B(\bar{v}, \gamma) \backslash \Gamma$. Moreover, the fourth inequality of~\eqref{ineq:4kinds}, $b > 1$, and $\gamma \in (0,1]$ lead to $\sigma_{k+1} \leq b \sigma_{k}^{2} \leq \frac{\gamma^{2}}{4b^{5}} \leq \frac{\gamma}{2b^{3}}$. Thus, we can verify that~\eqref{ineq:math_induction} holds for all $k \in \mathbb{N} \cup \{ 0 \}$. It then follows from the last assertion of~\eqref{ineq:4kinds} that problem P$(v_{k})$ has an approximate local optimum $w_{k}$ for each $k \in \mathbb{N} \cup \{ 0 \}$ if $v_{0} \in B_{V}(\bar{v}, r) \backslash \Gamma$. Therefore, Algorithm~\ref{Local_SSQP} is well defined whenever its initial point is selected from $B_{V}(\bar{v}, r) \backslash \Gamma$.
\par
We prove the quadratic convergence of an infinite sequence $\{ v_{k} \} \subset V \backslash \Gamma$ generated by Algorithm~\ref{Local_SSQP}, where the initial point satisfies $v_{0} \in B_{V}(\bar{v}, r) \backslash \{ 0 \}$. Now, we know that \eqref{ineq:math_induction} is valid for each $k \in \mathbb{N} \cup \{ 0 \}$. Then, thanks to $r \leq \gamma$ and the first result of~\eqref{ineq:math_induction}, any $v = v_{k}$ satisfies~\eqref{ineq:4kinds} for each $k \in \mathbb{N} \cup \{ 0 \}$. The first, second, and fourth results of~\eqref{ineq:4kinds}, the second result of~\eqref{ineq:math_induction}, and $\gamma \in (0,1]$ yield
\begin{align*}
\Vert v_{k} - \bar{v} \Vert_{V} \leq b \sigma_{k} \leq b^{2} \sigma_{k-1}^{2} \leq \frac{\gamma}{2 b} \sigma_{k-1} \leq \frac{\gamma}{2} \Vert v_{k-1} - \bar{v} \Vert_{V} \leq \frac{1}{2} \Vert v_{k-1} - \bar{v} \Vert_{V}
\end{align*}
for all $k \in \mathbb{N}$. By repeatedly exploiting the above result, we have $\Vert v_{k} - \bar{v} \Vert_{V} \leq \frac{1}{2^{k}} \Vert v_{0} - \bar{v} \Vert_{V}$ for all $k \in \mathbb{N}$. The fact means that $\{ v_{k} \}$ converges to $\bar{v}$. Let $m > 0$ be defined by $m \coloneqq b^{4}$. It then follows from the first, second, and fourth inequalities of~\eqref{ineq:4kinds} that for all $k \in \mathbb{N} \cup \{ 0 \}$,
\begin{align*}
\Vert v_{k+1} - \bar{v} \Vert_{V} \leq b \sigma_{k+1} \leq b^{2} \sigma_{k}^{2} \leq m \Vert v_{k} - \bar{v} \Vert_{V}^{2}.
\end{align*}
Therefore, the sequence $\{ v_{k} \}$ converges quadratically to $\bar{v}$.
\end{proof}

\section{Conclusions}
We proposed a stabilized SQP method for problem~\eqref{problem:main} and provided its local and fast convergence analysis. The considered problem has a general form that can represent a wide range of optimization problems, such as optimal control, obstacle, and shape optimization problems. Although the proposed method is classified as an SQP-type method, it differs from the ordinary SQP methods in the sense that the stabilized subproblems are iteratively solved. The stabilized subproblems are always feasible, but unlike the existing stabilized SQP methods for finite-dimensional optimization, their solvability is not ensured because this study considers infinite-dimensional optimization. Hence, it is difficult to exploit convergence analysis similar to those existing methods. However, under the SRCQ and the SOSC, we proved that the stabilized subproblems can be approximately solved, and their approximate solutions play a crucial role in the local convergence analysis. Consequently, the local and fast convergence of the proposed method was established in a different manner from the existing study on stabilized SQP methods for finite-dimensional optimization.
\par
In future research, it is worth proving the fast local convergence of stabilized SQP methods without assuming any CQ.


\bibliographystyle{abbrv}
\bibliography{references}

\end{document}